\documentclass[12pt]{amsart}
\usepackage{amsmath,amssymb,bm}
\usepackage{graphicx,color}
\usepackage{subfigure}
\newtheorem{theorem}{Theorem}[section]
\newtheorem{lemma}[theorem]{Lemma}
\newtheorem{corollary}[theorem]{Corollary}
\newtheorem{remark}[theorem]{Remark}

\begin{document}
\numberwithin{equation}{section}
\numberwithin{table}{section}
\numberwithin{figure}{section}
\allowdisplaybreaks[4]
\def\O{{\Omega}}
\def\R{\mathbb{R}}
\def\p{\partial}
\def\d{\displaystyle}
\def\PO{\Pi_{k,h}^{\raise 1pt\hbox{$\scriptscriptstyle\nabla$}}}
\def\POO{\Pi_{h,1}^{\raise 1pt\hbox{$\scriptscriptstyle\nabla$}}}
\def\POD{\Pi_{k,D}^{\raise 1pt\hbox{$\scriptscriptstyle\nabla$}}\hspace{1pt}}
\def\POOD{\Pi_{1,D}^{\raise 1pt\hbox{$\scriptscriptstyle\nabla$}}\hspace{1pt}}
\def\PDZ{\Pi_{k,D}^0\hspace{1pt}}
\def\PZ{\Pi_{k,h}^0}
\def\PZmt{\Pi_{k-2,h}^0}
\def\Qh{Q_h}
\def\cT{\mathcal{T}}
\def\cQ{\mathcal{Q}}
\def\cQD{\cQ^k(D)}
\def\cQbD{\cQ^k(\p D)}
\def\cQh{\cQ^k_h}
\def\cP{\mathcal{P}}
\def\cPh{\cP^k_h}
\def\bbP{\mathbb{P}}
\def\PkD{\mathbb{P}_k(D)}
\def\Pk{\mathbb{P}_k}
\def\PkbD{\mathbb{P}_k(\p D)}
\def\PkmtD{\mathbb{P}_{k-2}(D)}
\def\PDZmt{\Pi_{k-2,D}^0}
\def\ID{I_{k,D}}
\def\tbar{|\!|\!|}
\def\hD{\hspace{1pt}h_{\ssD}}
\def\hE{h_{E}}
\def\Ih{I_{k,h}}
\def\fB{\mathfrak{B}}
\def\cE{\mathcal{E}}
\def\KerD{\mathcal{N}(\POD)}
\def\ssD{{\scriptscriptstyle D}}
\def\cD{c_{\ssD}}
\def\rD{\rho_{\!{_\ssD}}}
\def\ED{\cE_{\ssD}}
\def\tD{\tau_{\!\ssD}}
\def\hS{h_{\sigma}}
\def\hE{h_e}
\def\hF{h_F}
\def\fBD{\fB_{\!\ssD}}
\def\tBD{\tilde\fB_{\!\ssD}}
\def\ET{\mathbb{E}_{\ssD}}
\def\TL{\mathrm {Tr}^\dag}
\def\SumD{\sum_{D\in\cT_h}}
\def\NPD{\mathcal{N}_{\p D}}
\def\unorm{|u|_{H^{\ell+1}(\O)}}
\def\LTD{{L_2(D)}}
\def\HOD{{H^1(D)}}
\def\HTD{{H^2(D)}}
\def\LinfD{{L_\infty(D)}}
\def\POF{\Pi_{k,F}^{\raise 1pt\hbox{$\scriptscriptstyle\nabla$}}\hspace{1pt}}
\def\LTF{{L_2(F)}}
\def\ssF{\scriptscriptstyle{F}}
\def\EF{\cE_{\ssF}}
\def\FD{\mathcal{F}_\ssD}
\def\SumF{\sum_{F\in\FD}}
\def\NPF{\mathcal{N}_{\p F}}
\def\PFZ{\Pi_{k,F}^0}
\def\PFZmo{\Pi_{k-1,F}^0}
\def\PFZmt{\Pi_{k-2,F}^0}
\def\tF{\tau_{\!\ssF}}
\def\cF{\mathcal{F}}
\def\factor{\ln(1+\tD)}
\def\IF{I_{k,F}}
\def\cN{\mathcal{N}}
\def\HOF{{H^1(F)}}
\def\factorTD{\ln(1+\max_{F\in\cF_D}\tF)}
\def\GF{\nabla_{\!\!F}}
\title[Virtual Element Methods on Meshes with Small Edges or Faces]
{Virtual Element Methods on Meshes\\ with Small Edges or Faces}
\author[S.C. Brenner]{Susanne C. Brenner}
\address{Susanne C. Brenner, Department of Mathematics and Center for
Computation and Technology, Louisiana State University,
Baton Rouge, LA 70803} \email{brenner@math.lsu.edu}
\author[L.-Y. Sung]{Li-yeng Sung} \address{Li-yeng Sung,
 Department of Mathematics and Center for Computation and Technology,
 Louisiana State University, Baton Rouge, LA 70803}
\email{sung@math.lsu.edu}
\thanks{This work  was supported in part
 by the National Science Foundation under Grant No. DMS-16-20273.}
\begin{abstract} We consider a model Poisson problem in $\R^d$ ($d=2,3$)
 and establish error estimates
 for virtual element methods on polygonal or polyhedral meshes that
 can contain
 small edges ($d=2$) or small faces ($d=3$).
\end{abstract}
%\date{September 14, 2017}
\maketitle
%
%%%%%%%%%%%%%%%%%%%%%%%%%%%%%%%%%%%%%%%%%%%%%%%%%%%%%%%%
\section{Introduction}\label{sec:Introduction}
 Let $\O\subset \R^d$ ($d=2,3$) be a  bounded polygonal/polyhedral domain
 and $f\in L_2(\O)$. The Poisson problem
 with the homogeneous Dirichlet boundary condition is to find
 $u\in H^1_0(\O)$ such that
\begin{equation}\label{eq:Poisson}
 a(u,v)=(f,v) \qquad\forall\,v\in H^1_0(\O),
\end{equation}
 where
\begin{equation}\label{eq:aDef}
  a(u,v)=\int_\O \nabla u\cdot\nabla v\,dx
\end{equation}
 and $(\cdot,\cdot)$ is the  inner product for $L_2(\O)$.
 Here and throughout the paper we follow standard notation for differential
 operators, function
 spaces and norms that can be found for example in
 \cite{Ciarlet:1978:FEM,ADAMS:2003:Sobolev,BScott:2008:FEM}.
\par
 Problem \eqref{eq:Poisson} can be solved by virtual element methods
 \cite{BBCMMR:2013:VEM,AABMR:2013:Projector} on
 polygonal or polyhedral meshes.
 It has been observed in numerical experiments that the convergence rates
 for the virtual element methods do not deteriorate noticeably even
 in the presence of small edges or faces
 (cf. \cite{AABMR:2013:Projector,BLR:2016:VEM,BDR:2017:VEM}).
 Our goal is to establish error estimates that justify these numerical results
 for the virtual element methods introduced in \cite{AABMR:2013:Projector}.
\par
 We will develop error estimates that are based on general shape regularity
 assumptions on the subdomains in
 the polygonal or polyhedral meshes.  For the two dimensional problem, we assume
 that (i) each polygonal subdomain  is
 star-shaped with respect to a disc whose diameter is comparable to the
 diameter of the subdomain and (ii)
 the number of edges of the subdomains is uniformly bounded.  For the three
 dimensional problem, we assume
 that (i) each polyhedral subdomain is star-shaped with respect to a ball
  whose diameter is comparable to the diameter of the subdomain;
  (ii) the number of faces of the subdomains is uniformly bounded; and
 (iii) the faces of the subdomains satisfy the two dimensional shape
 regularity assumptions.
 Our error estimates  are optimal up to at most a logarithmic
 factor that involves the ratio of the lengths of the longest edge and
 the shortest edge of each subdomain
  in two dimensions and a similar ratio over the edges of the faces on the
  subdomains in three dimensions.
\par
 The rest of the paper is organized as follows.  We  begin with a
 star-shaped condition
 in Section~\ref{sec:SS}.  Then we treat the two dimensional case in
 Section~\ref{sec:LocalVEM2D} and Section~\ref{sec:Poisson2D},
 where the analysis benefits from the
 techniques developed in \cite{BLR:2016:VEM} and \cite{BGS:2017:VEM2}.
 The extension to the three dimensional Poisson problem is presented in
 Section~\ref{sec:Poisson3D}.
 We end with some concluding remarks in Section~\ref{sec:Conclusions}.
\par
 In order to avoid the proliferation of constants, we will often use the notation
 $A\lesssim B$ to represent the statement that $A\leq (\text{constant}) B$,
 where the positive constant is independent
 of mesh sizes.  The notation $A\approx B$ is equivalent to $A\lesssim B$
 and $B\lesssim A$.
 The precise dependence of the hidden constants will be declared in the text.
\par
 To minimize the technicalities, we also assume that $\O$ is convex so that the
 solution of \eqref{eq:Poisson} belongs to $H^2(\O)$ by
 elliptic regularity \cite{Grisvard:1985:EPN,Dauge:1988:EBV}.
%
%%%%%%%%%%%%%%%%%%%%%%%%%%%%%%%%%%%%%%%%%%%%%%%%%%
\section{A Star-Shaped Condition}\label{sec:SS}
 Let $D$ be a bounded open polygon ($d=2$) or a bounded open
 polyhedron $(d=3)$, and
 $\hD$ be the diameter of $D$.
\par
 We assume that
\begin{equation}\label{eq:SA}
  \text{$D$ is star-shaped with respect to a disc/ball $\fBD\subset D$ with radius $\rD\hD$}.
\end{equation}
\par
 We will denote by $\tBD$ the disc/ball concentric with $\fBD$ whose radius is $\hD$.
 It is clear that
\begin{equation}\label{eq:discs}
  \fBD\subset D\subset \tBD.
\end{equation}
 \par
 Below are some consequences of the star-shaped condition \eqref{eq:SA}.
 The hidden constants in Section~\ref{subsec:Sobolev}--Section~\ref{subsec:PF}
 only depend on $\rD$, while those in Section~\ref{subsec:DE} also depend on $k$.
%
%%%%%%%%%%%%%%%%%%%%%%%%%%%%%%%%%%%%%%%
\subsection{Sobolev Inequalities}\label{subsec:Sobolev}
 It follows from \eqref{eq:SA} that
\begin{alignat}{3}
    \|\zeta\|_{\LinfD}&\lesssim  \hD^{-(d/2)}\|\zeta\|_\LTD+\hD^{1-(d/2)}|\zeta|_{\HOD}
      +\hD^{2-(d/2)}|\zeta|_{H^2(D)}&\qquad&\forall\,
      \zeta\in H^2(D),\label{eq:Sobolev}\\
 \intertext{and in the case where $d=2$,}
      \|\zeta\|_{\LinfD}&\lesssim  \hD^{-1}\|\zeta\|_\LTD+|\zeta|_{\HOD}
      +\hD^{1/2}|\zeta|_{H^{3/2}(D)}&\qquad&\forall\,
      \zeta\in H^{3/2}(D).\label{eq:Sobolev2}
\end{alignat}
 Details can be found in \cite[Lemma~4.3.4]{BScott:2008:FEM} and
 \cite[Section~6]{DS:1980:BH}.
%%%%%%%%%%%%%%%%%%%%%%%%%%%%%%%%
%
\subsection{Bramble-Hilbert Estimates}\label{subsec:BH}
 Condition \eqref{eq:SA} also implies the following Bramble-Hilbert
 estimates \cite{BH:1970:Lemma}:
%\par We have the following Bramble-Hilbert estimates \cite{BH:1970:Lemma}:
%
\begin{alignat}{3}
  \inf_{q\in\bbP_\ell}|\zeta-q|_{H^m(D)}&\lesssim
   h_D^{\ell+1-m}|\zeta|_{H^{\ell+1}(D)}
  &\qquad&\forall\,\zeta\in H^{\ell+1}(D),
  \,\ell=0,\ldots,k,\;\text{and}\; m\leq \ell,
  \label{eq:BHEstimates}\\
   \inf_{q\in\bbP_\ell}|\zeta-q|_{H^m(D)}&
   \lesssim h_D^{\ell+\frac12-m}|\zeta|_{H^{\ell+\frac12}(D)}
  &\qquad&\forall\,\zeta\in H^{\ell+\frac12}(D),
   \,\ell=0,\ldots,k,\;\text{and}\; m\leq \ell.
  \label{eq:BHEstimates2}
\end{alignat}
 Details can be found in \cite[Lemma~4.3.8]{BScott:2008:FEM}  and
 \cite[Section~6]{DS:1980:BH}.
%
%%%%%%%%%%%%%%%%%%%%%%%%%%%%%%%%%%
\subsection{A Lipschitz Isomorphism between $D$ and $\fBD$}\label{subsec:Lipschitz}
\par
 In view of the star-shaped condition \eqref{eq:SA},
 there exists a Lipschitz isomorphism $\Phi:\fBD\longrightarrow D$ such that
 both $|\Phi|_{W^{1,\infty}(\fBD)}$ and $|\Phi^{-1}|_{W^{1,\infty}(D)}$
 are bounded by a constant that
 only depends on $\rD$ (cf.  \cite[Section~1.1.8]{Mazya:2011:Sobolev}).
\par
 It follows that
 \begin{align}
  |D|&\approx \hD^d,\label{eq:G1}\\
  |\p D|&\approx \hD^{d-1},\label{eq:G2}
\end{align}
 where $|D|$ is the area of $D$ ($d=2$) or the volume of $D$ ($d=3$), and $|\p D|$ is the
 arclength  of $\p D$ ($d=2$) or the surface area of $D$ ($d=3$).
\par
 Moreover we have  (cf. \cite[Theorem~4.1]{Wloka:1987:PDE})
\begin{alignat}{3}
  \|\zeta\|_{L_2(\p D)}&\approx \|\zeta\circ\Phi\|_{L_2(\p\fBD)}
  &\qquad&\forall\,\zeta\in L_2(\p D),\label{eq:L2Iso1}\\
  \|\zeta\|_\LTD&\approx \|\zeta\circ\Phi\|_{L_2(\fBD)}
  &\qquad&\forall\,\zeta\in L_2(D),\label{eq:L2Iso2}\\
  |\zeta|_{H^1(\p D)}&\approx|\zeta\circ\Phi|_{H^1(\p\fBD)}
  &\qquad&\forall\,\zeta\in H^1(\p D),\label{eq:H1Iso1}\\
  |\zeta|_{\HOD}&\approx|\zeta\circ\Phi|_{H^1(\fBD)}
  &\qquad&\forall\,\zeta\in \HOD,\label{eq:H1Iso2}\\
  |\zeta|_{H^{1/2}(\p D)}&\approx |\zeta\circ\Phi|_{H^{1/2}(\p\fBD)}
  &\qquad&\forall\,\zeta\in H^{1/2}(\p D).
  \label{eq:HalfIso}
 \end{alignat}
%
 %%%%%%%%%%%%%%%%%%%%%%%%%%%%%%%%%%%%%
\subsection{Poincar\'e-Friedrichs Inequalities}\label{subsec:PF}
 The Bramble-Hilbert estimate \eqref{eq:BHEstimates} and the geometric estimates
 \eqref{eq:G1}--\eqref{eq:G2}  imply the following Poincar\'e-Friedrichs inequalities:
\begin{alignat}{3}
 h_D^{-(d/2)} \|\zeta\|_\LTD&\lesssim h_D^{-d}\Big|\int_{D} \zeta\,dx\Big|
 +\hD^{1-(d/2)}|\zeta|_{\HOD}
    &\qquad&\forall\,\zeta\in \HOD, \label{eq:PF1}\\
    h_D^{-(d/2)} \|\zeta\|_\LTD&\lesssim h_D^{-(d-1)}\Big|\int_{\p D} \zeta\,ds\Big|
    +\hD^{1-(d/2)}|\zeta|_{\HOD}
    &\qquad&\forall\,\zeta\in \HOD.\label{eq:PF2}
\end{alignat}
%
 %%%%%%%%%%%%%%%%%%%%%%%%%
\subsection{Estimates for $|\cdot|_{H^{1/2}(\p D)}$}\label{subsec:Half}
 On the circle $\p \fBD$, we have a standard estimate
\begin{equation*}
   |\zeta|_{H^{1/2}(\p\fBD)}\lesssim \hD^{-1/2}\|\zeta\|_{L_2(\p\fBD)}
   +\hD^{1/2}|\zeta|_{H^1(\p \fBD)}
   \qquad\forall\,\zeta\in H^1(\p \fBD).
\end{equation*}
  It then follows from a Poincar\'e-Friedrichs inequality
  for a circle that
\begin{equation*}
  |\zeta|_{H^{1/2}(\p\fBD)}=|\zeta-\bar\zeta|_{H^{1/2}(\p \fBD)}
        \lesssim \hD^{-1/2}\|\zeta-\bar\zeta\|_{L_2(\p\fBD)}
        +\hD^{1/2}|\zeta|_{H^1(\p \fBD)}
        \lesssim \hD^{1/2}|\zeta|_{H^1(\p\fBD)},
\end{equation*}
  where $\bar\zeta$ is the mean of $\zeta$ over $\p \fBD$.
  Therefore, in view of \eqref{eq:H1Iso1} and \eqref{eq:HalfIso}, we have
 \begin{equation}\label{eq:HalfAndOne}
    |\zeta|_{H^{1/2}(\p D)}\lesssim \hD^{1/2}|\zeta|_{H^1(\p D)}
    \qquad\forall\,\zeta\in H^1(\p D).
\end{equation}
\par
 Similarly, it follows from \eqref{eq:H1Iso2}, \eqref{eq:HalfIso}
 and the trace theorem for $H^1(\fBD)$ that
\begin{alignat}{3}
  |\zeta|_{H^{1/2}(\p D)}&\lesssim |\zeta|_{\HOD}
  &\qquad&\forall\,\zeta\in \HOD.\label{eq:HalfTrace}
\end{alignat}
%
%%%%%%%%%%%%%%%%%%%%%%%%%%%%%%%%%%%%%%%%%%%%%%%%
\subsection{Trace Inequalities}\label{subsec:Trace}
  It follows from \eqref{eq:L2Iso1}, \eqref{eq:L2Iso2}, \eqref{eq:H1Iso2}
  and  standard (scaled) trace inequalities
  for $H^1(\fBD)$    that
\begin{alignat}{3}
  \|\zeta\|_{L_2(\p D)}^2&\lesssim \hD^{-1}\|\zeta\|_\LTD^2+\hD|\zeta|_{\HOD}^2
  &\qquad&\forall\,\zeta\in \HOD.\label{eq:Trace}
\end{alignat}
\par
 We also have trace inequalities for the $H^1$ norm on $\p D$ that require a different derivation.
\begin{lemma}\label{lem:CZ1}
  Let $e$ be an edge of $D\subset \R^2$.  We have
\begin{equation*}%\label{eq:CZ1}
  \hD|\zeta|_{H^1(e)}^2\lesssim |\zeta|_\HOD^2+\hD|\zeta|_{H^{3/2}(D)}^2 \qquad
  \forall\,\zeta\in H^{3/2}(D).
\end{equation*}
\end{lemma}
\begin{proof} By scaling we can assume $\hD=1$. Without
loss of generality we may also assume that
\begin{equation}\label{eq:CZ1Est0}
   \int_D \zeta\,dx=0.
\end{equation}
\par
  The existence of the Lipschitz isomorphism $\Phi:D\longrightarrow \fBD$ implies that the domain
 $D$ satisfies
 a uniform cone condition \cite[Section~4.8]{ADAMS:2003:Sobolev},
 with one reference cone and a finite cover of $\bar D$ that contains a
 fixed number of congruent open discs.  Furthermore, the angle and the height of the
 reference cone and the radius of the open discs only depend on $\rD$.  It follows that there
 exists a Calderon-Zygmund extension operator
 $E:H^1(D)\longrightarrow H^1(\R^2)$
 (cf. \cite[Theorem~5.28]{ADAMS:2003:Sobolev}) such that
 $E$ maps $H^2(D)$ into $H^2(\R^2)$ and the restriction of
 $E\zeta$ to $D$ equals $\zeta$.  Moreover we have
\begin{equation*}
  \|E\zeta\|_{H^1(\R^2)}\lesssim \|\zeta\|_\HOD \quad\forall\,\zeta\in\HOD
  \quad\text{and}\quad
  \|E\zeta\|_{H^2(\R^2)}\lesssim \|\zeta\|_\HTD \quad\forall\,\zeta\in\HTD.
\end{equation*}
 It follows from the interpolation of Sobolev spaces \cite[Chapter~7]{ADAMS:2003:Sobolev}
 that
\begin{align}\label{eq:CZ1Est1}
  \|E\zeta\|_{H^{3/2}(\R^2)}\lesssim
  \|\zeta\|_{H^{3/2}(D)}
  \lesssim  |\zeta|_\HOD
  +|\zeta|_{H^{3/2}(D)}
  \qquad\forall\,\zeta\in H^{3/2}(D),
\end{align}
 where we have used \eqref{eq:PF2} and \eqref{eq:CZ1Est0}.
\par
 Let $e$ be an edge of $D$, $\tilde e$ be the infinite line that contains $e$
  and $G$ be a half-plane that borders $\tilde e$.
 Then we have, by \eqref{eq:CZ1Est1} and
 the trace theorem \cite[Theorem~8.1]{Wloka:1987:PDE},
\begin{align*}%\label{eq:CZ1Est2}
 |\zeta|_{H^1(e)}&=|E\zeta|_{H^1(e)}
          \leq |E\zeta|_{H^1(\tilde e)}
          \lesssim \|E\zeta\|_{H^{3/2}(G)}
          \lesssim  |\zeta|_\HOD+|\zeta|_{H^{3/2}(D)}.
\end{align*}
\end{proof}
\par
 The proof of the following result is similar.
\begin{lemma}\label{lem:CZ2}
  Let $F$ be a face of $D\subset \R^3$.  We have
\begin{equation*}%\label{eq:CZ2}
  \hD|\zeta|_{H^1(F)}^2\lesssim |\zeta|_\HOD^2+\hD^{2}|\zeta|_\HTD^2 \qquad
  \forall\,\zeta\in H^{2}(D).
\end{equation*}

\end{lemma}

%%%%%%%%%%%%%%%%%%%%
\subsection{A Lifting Operator}\label{subsec:Lifting}
 It follows from \eqref{eq:L2Iso2}, \eqref{eq:H1Iso2}, \eqref{eq:HalfIso} and the
 inverse trace theorem for $H^1(\fBD)$
 (cf. \cite[Theorem~8.8]{Wloka:1987:PDE}) that
 there exists a linear operator $\TL:H^{1/2}(\p D)\longrightarrow \HOD$
 such that
\begin{equation*}%\label{eq:Lifting}
 \TL \zeta=\zeta\;\text{on}\;\p D
 %\int_D \TL \zeta\,dx=0
 \quad\text{and}\quad
 \hD^{-1}\|\TL \zeta\|_\LTD+|\TL \zeta|_{\HOD}\leq C|\zeta|_{H^{1/2}(\p D)},
\end{equation*}
 where the constant $C$ depends only on $\rD$.
 %
%%%%%%%%%%%%%%%%%%%%%%%%%%%%%%%%%%%%%%%
\subsection{Some Estimates for Polynomials}\label{subsec:DE}
 Let $\bbP_k$ be the space of polynomials of total degree $\leq k$ in $d$ variables.
 We obtain the following estimates by using the equivalence of norms
 on finite dimensional vector spaces and scaling.
\begin{lemma}\label{lem:DiscreteEstimates}
 We have
\begin{alignat}{3}
   \|p\|_{L_2(\p D)}^2&\lesssim \hD^{-1}\|p\|_\LTD^2&\qquad&\forall\,p\in\bbP_k,
   \label{eq:DiscreteEstimate0}\\
    |p|_{\HOD}&\lesssim  h_D^{-1}\|p\|_\LTD
      &\qquad&\forall\,p\in \bbP_k,
      \label{eq:DiscreteEstimate1}\\
      \|p\|_{L_\infty(D)}&\lesssim |\bar{p}_{_{\p D}}|+\hD^{1-(d/2)}
      |p|_{\HOD}&\qquad&\forall\,p\in\bbP_k,
       \label{eq:DiscreteEstimate2}\\
          \|p\|_{L_\infty(D)}&\lesssim |\bar{p}_{_{D}}|+\hD^{1-(d/2)}
          |p|_{\HOD}&\qquad&\forall\,p\in\bbP_k,
       \label{eq:DiscreteEstimate3}
\end{alignat}
 where $\bar p_{_{\p D}}$ is the mean of $p$ over $\p D$ and
 $\bar p_{_D}$ is the mean of $p$ over $D$.
\end{lemma}
\begin{proof}
 In view of \eqref{eq:discs}, we have
\begin{align}\label{eq:DS2}
 \|p\|_{L_\infty(D)}\leq \|p\|_{L_\infty(\tilde\fB_{\ssD})}
 &\lesssim \|p\|_{L_\infty(\fB_{\ssD})}\\
 &\lesssim (\text{diam}\,\fB_{\ssD})^{-d/2}\|p\|_{L_2(\fB_{\ssD})}
 \leq (\text{diam}\,\fB_{\ssD})^{-d/2}\|p\|_\LTD.\notag
\end{align}
 The estimate \eqref{eq:DiscreteEstimate0}
 then follows from \eqref{eq:G2} and \eqref{eq:DS2}:
\begin{align*}
  \|p\|_{L_2(\p D)}^2
                     \lesssim \hD^{d-1}\|p\|_{L_\infty(D)}^2
                     \lesssim \hD^{-1}\|p\|_\LTD^2.
\end{align*}
\par
 Similarly, we have
\begin{equation*}
  |p|_{\HOD}\leq |p|_{H^1(\tilde\fB_{\ssD})}\lesssim |p|_{H^1(\fB_{\ssD})}
  \lesssim (\text{diam}\,\fB_{\ssD})^{-1}\|p\|_{L_2(\fB_{\ssD})}\leq
  (\text{diam}\,\fB_{\ssD})^{-1}\|p\|_\LTD,
\end{equation*}
 which together with \eqref{eq:SA} implies \eqref{eq:DiscreteEstimate1}.
\par
 The estimates \eqref{eq:DiscreteEstimate2} and \eqref{eq:DiscreteEstimate3}
 follow immediately from \eqref{eq:G1}--\eqref{eq:G2}, \eqref{eq:PF1}--\eqref{eq:PF2}  and
 \eqref{eq:DS2}.
\end{proof}
\begin{lemma}\label{lem:Polynomial}
  Given any $p\in \bbP_{k-2}$ $(k\geq2)$, there exists $q\in\bbP_k$ such that
  $\Delta q=p$ and
\begin{equation*}
    \|\nabla q\|_{L_2(B)}\leq C (\mathrm{diam}\, B)\|p\|_{L_2(B)},%\label{eq:Polynomial}
\end{equation*}
 where $B\subset \R^d$ is any ball and the positive constant $C$ depends only on $k$.
\end{lemma}
\begin{proof}  By scaling
  it suffices to treat the case where $B$ is a unit ball.
  Since $\Delta$ maps $\Pk$ onto $\bbP_{k-2}$, there exists an operator
 $\Delta^\dag:\bbP_{k-2}\longrightarrow \Pk$ such that $\Delta \Delta^\dag$ is the
 identity operator on $\bbP_{k-2}$, and we can take $q=\Delta^\dag p$.
   The lemma follows from the observation that
 both $\|p\|_{L_2(B)}$ and $\|\Delta^\dag p\|_{L_2(B)}$ are norms on
 $\bbP_{k-2}$ together with a standard inverse estimate
 \cite{Ciarlet:1978:FEM,BScott:2008:FEM}.
\end{proof}
\par

%%%%%%%%%%%%%%%%%%%%%%%%%%
\section{Local Virtual Element Spaces in Two Dimensions}\label{sec:LocalVEM2D}
%%%%%%%%%%%%%%%%%%%%%%%%%%%%%%%%
%
 In this section we obtain properties of the local virtual element spaces
 that will be used in the stability
 and error analyses in Section~\ref{sec:Poisson2D}.
\par
   Let the space $\PkD$ be the restriction of $\bbP_k$
 to $D$ and the space of $\PkbD$ be defined by
\begin{equation*}
  \PkbD=\{v\in C(\p D):\,v\big|_e\in \bbP_k(e) \;\text{for all}\;e\in\ED\},
\end{equation*}
 where $C(\p D)$ is the space of continuous functions on $\p D$,
 $\bbP_k(e)$
 is the restriction of $\bbP_k$ to the edges $e$, and $\ED$ is the set
 of the edges of $D$.  The length of an edge $e$ is denoted by $\hE$.
%
%%%%%%%%%%%%%%%%%%%%%%%%%%%%%%%%%%%%%%%%%%%%%%
%
\subsection{The Projection $\POD$ and the Space $\cQD$}\label{subsec:PODCQD}
 The Sobolev space $\HOD$ is a Hilbert space under the inner product
\begin{equation}\label{eq:InnerProduct}
 (\!(\zeta,\eta)\!)=(\nabla \zeta,\nabla \eta)
     +\Big(\int_{\p D}\zeta\,ds\Big)\Big(\int_{\p D} \eta\,ds\Big).
\end{equation}
 The projection operator  from $\HOD$ onto $\PkD$ with respect
 to $(\!(\cdot,\cdot)\!)$ is denoted by $\POD$, i.e.,
 $\POD \zeta\in \PkD$ satisfies
\begin{equation}\label{eq:PODDef}
  (\!(\POD\zeta,q)\!)=(\!(\zeta,q)\!) \qquad\forall\,q\in\PkD.
\end{equation}
 In particular we have
\begin{equation}\label{eq:Projection}
 \POD p=p \qquad\forall\,p\in\PkD.
\end{equation}
\par
 It is straight-forward to check that \eqref{eq:PODDef} is equivalent to
\begin{alignat}{3}
   \int_D \nabla (\POD\zeta)\cdot\nabla q\,dx&=\int_D
  \nabla \zeta\cdot \nabla q\,dx\label{eq:POD1}\\
   &=\int_{\p D} \zeta(n\cdot\nabla q)\,ds-\int_D \zeta(\Delta q)\,dx
  &\qquad&\forall\,q\in \PkD,\notag
\end{alignat}
 together with
\begin{alignat}{3}
   \int_{\p D} \POD\zeta\,ds&=\int_{\p D}\zeta\,ds.
   \label{eq:POD2}
\end{alignat}
\par
 For $k\geq 1$,
 the virtual element space $\cQD\subset \HOD$ is defined by the
 following conditions:  $v\in \HOD$ belongs to $\cQD$ if and only if
 (i) the trace of $v$ on $\p D$ belongs to $\PkbD$,
  (ii) the distribution $-\Delta v$ belongs to $\PkD$,
 and (iii) we have
\begin{equation}\label{eq:Condition3}
 \PDZ v-\POD v\in \PkmtD,
\end{equation}
 where $\PDZ$ is the projection from $L_2(D)$ onto $\PkD$ and
 $\bbP_{-1}(D)=\{0\}$.
\begin{remark}\label{rem:Continuity}{\rm
 It follows from elliptic regularity for bounded Lipschitz domains
 \cite[Section~1.2]{Kenig:1994:CBMS}
 and conditions (i) and (ii) in the definition of $\cQD$ that
 $\cQD\subset C(\bar D)$.}
\end{remark}
\begin{remark}\label{rem:dof}{\rm The dimension of $\cQD$ is
 the sum of the dimension of $\PkbD$ and
 the dimension of $\PkmtD$ (cf. \cite{AABMR:2013:Projector}).
 The degrees of freedom consist of (i) the values of $v$ at
 the vertices of $D$ and
 at the points in the interior of the edges  that  together determine
 $\PkbD$,  and (ii) the moments of
 $\PDZmt v$.  The set of the nodes in (i) will be denoted by $\NPD$.}
\end{remark}
\begin{remark}\label{rem:Computable}
{\rm
 It follows from \eqref{eq:POD1} and \eqref{eq:POD2}
  that the polynomial $\POD v$ can be computed in terms of the
  degrees of freedom of
  $v\in\cQD$.  Moreover, the polynomial
  $\PDZ v$ can also be computed through \eqref{eq:Condition3}.}
\end{remark}
%%%%%%%%%%%%%%%%%%
\subsection{A Minimum Energy Principle}\label{subsec:MEP}
 The following minimum energy principle is useful for bounding the $H^1$
 norm of a virtual element function.
\begin{lemma}\label{lem:MEP}
   The inequality
\begin{equation*}%\label{eq:MEP}
  |v|_{\HOD}\leq |\zeta|_{\HOD}
\end{equation*}
 holds for any $v\in\cQD$ and $\zeta\in \HOD$ such that $\zeta-v=0$
 on $\p D$ and $\PDZ(\zeta-v)=0$.
\end{lemma}
\begin{proof}  It follows from condition (ii) in the definition of $\cQD$ that
\begin{equation*}
   \int_D \nabla v\cdot\nabla (\zeta-v)\,dx=\int_D (-\Delta v) (\zeta-v)dx=0
\end{equation*}
 and hence
 $
   |\zeta|_{H^1(D)}^2=|\zeta-v|_{H^1(D)}^2+|v|_{H^1(D)}^2.
 $
\end{proof}
%
%%%%%%%%%%%%%%%%%%%%%%%%%%%%%%
\subsection{A Maximum Principle}\label{subsec:MP}
 We begin with a result from \cite[Lemma~3.3]{BLR:2016:VEM}.
\begin{lemma}\label{lem:BLR}
   There exists a positive constant $C$, depending only on
   $\rD$ and $k$, such that
\begin{equation}\label{eq:BLR}
 \|\Delta v\|_\LTD\leq C\hD^{-1}\|\nabla v\|_\LTD \qquad\forall\,v\in\cQD.
\end{equation}
\end{lemma}
\begin{proof} By scaling we may assume $\hD=1$.
\par
  Let $\phi\geq0$ be a smooth (bump) function supported on the disc
  $\fBD$ with radius $\rD$
  such that
\begin{equation}\label{eq:BumpFunction}
  \int_D \phi\,dx=1, \quad |\phi|\lesssim 1 \quad \text{and}
  \quad |\nabla\phi(x)|\lesssim 1.
\end{equation}
 We have,
 by  the equivalence of norms on finite dimensional vector spaces, scaling
 and \eqref{eq:discs},
\begin{equation}\label{eq:Fundamental1}
  \|p\|_\LTD^2\leq  \|p\|_{L_2(\tBD)}^2\lesssim \|p\|_{L_2(\fBD)}^2
  \lesssim \int_{\fBD} p^2\phi\,dx
  \qquad\forall\,p\in\bbP_k.
\end{equation}
\par
 Since $\Delta v\in\bbP_k$,  it follows from \eqref{eq:DiscreteEstimate1}
 (with $\hD=1$), \eqref{eq:Fundamental1}
 and integration by parts that
\begin{align*}
  \|\Delta v\|_\LTD^2&\lesssim \int_{\fBD} (\Delta v)^2\phi\,dx\\
    &= -\int_{\fBD} \nabla v\cdot(\phi\nabla(\Delta v)+(\Delta v)\nabla\phi\big)dx\\
          &\lesssim \|\nabla v\|_\LTD \big(\|\nabla(\Delta v)\|_\LTD+\|\Delta v\|_\LTD\big)
           \lesssim  \|\nabla v\|_\LTD\big(\|\Delta v\|_\LTD\big),
\end{align*}
 which implies \eqref{eq:BLR} (with $\hD=1$).
\end{proof}
  The following maximum principle will be used in the analysis of the
  interpolation operator in Section~\ref{subsec:Interpolation}
  (cf. Lemma~\ref{lem:1DInterpolationError}),
  and in the stability and error analyses for
  virtual element methods in three dimensions (cf. \eqref{eq:3DIDtbar}
  and Lemma~\ref{lem:3DSDBdd}).
\begin{lemma}\label{lem:MaximumPrinciple}
  There exists a positive constant $C$, depending only on $\rD$ and $k$, such that
\begin{equation*}
  \|v\|_{\LinfD}\leq C\big[ \|v\|_{L_\infty(\p D)}+|v|_\HOD\big]
  \qquad\forall\,v\in \cQD.
\end{equation*}
\end{lemma}
\begin{proof}
 There exists $q\in \bbP_{k+2}$ such that
\begin{equation}\label{eq:MP2}
 \Delta q=\Delta v \quad\text{and}\quad \|\nabla q\|_{L_2(\fBD)}
 \lesssim \hD\|\Delta v\|_{L_2(\fBD)}
\end{equation}
 by Lemma~\ref{lem:Polynomial} (with $p=\Delta v\in\bbP_k$).
\par
 Without loss of generality we may
 assume that the mean of $q$ over $D$ is zero.  Therefore we have
\begin{equation*}%\label{eq:MP3}
 \|q\|_{\LinfD}\lesssim |q|_{\HOD}
       \lesssim |q|_{H^1(\tBD)}
       \lesssim |q|_{H^1(\fBD)}
       \lesssim \hD\|\Delta v\|_{L_2(\fBD)}
      % \lesssim \|\nabla v\|_{L_2(\fBD)}
       \lesssim |v|_{\HOD}\notag
\end{equation*}
 by \eqref{eq:discs}, \eqref{eq:DiscreteEstimate3}, Lemma~\ref{lem:BLR},
 \eqref{eq:MP2} and scaling.
\par
 It then follows from %\eqref{eq:MP3} and
 the maximum principle for the harmonic function $v-q$
 (cf. \cite{Evans:2010:PDE}) that
\begin{align*}
 \|v\|_{\LinfD}&\leq \|v-q\|_{\LinfD}+\|q\|_{\LinfD}\\
   &\leq \|v-q\|_{L_\infty(\p D)}+\|q\|_{\LinfD}
       \lesssim \|v\|_{L_\infty(\p D)}+|v|_\HOD.
\end{align*}
\end{proof}
%%%%%%%%%%%%%%%%%%%%%%%%%%%%%%%%%%%%%%%
\subsection{The Semi-norm ${\tbar\cdot\tbar_{k,D}}$}\label{subsec:Norm}
 The semi-norm $\tbar\cdot\tbar_{k,D}$ on $\HOD$ is defined by
\begin{equation}\label{eq:tbarNormDef}
  \tbar \zeta\tbar_{k,D}^2=\|\Pi_{k-2,D}^0\zeta\|_\LTD^2+
  \hD\sum_{e\in\ED}\|\Pi_{k-1,e}^0 \zeta\|_{L_2(e)}^2,
 \end{equation}
 where
 $\Pi_{k-1,e}^0$ is the orthogonal projection from $L_2(e)$ onto $\bbP_{k-1}(e)$.
\par
 It follows from \eqref{eq:Trace} and \eqref{eq:tbarNormDef} that
\begin{equation}\label{eq:tbarBdd}
  \tbar\zeta\tbar_{k,D}\leq C\big(\|\zeta\|_\LTD+\hD|\zeta|_{\HOD}\big)
  \qquad\forall\,
  \zeta\in \HOD,
\end{equation}
 where the positive constant $C$ depends only on $\rD$ and $k$.
\par
 We also have, by \eqref{eq:G2} and a standard estimate for polynomials
 in one variable,
\begin{align}\label{eq:tbarBdd2}
 \tbar v\tbar_{k,D}&\lesssim \hD\|v\|_{L_\infty(\p D)}
 +\|\Pi_{k-2,D}^0 v\|_\LTD\\
    &\lesssim  \hD\Big(\sum_{p\in\NPD} v^2(p)\Big)^\frac12
    +\|\Pi_{k-2,D}^0 v\|_\LTD
   \qquad \forall\,v\in\cQD,\notag
\end{align}
 where the hidden constant depends only on $\rD$ and $k$.
%
%%%%%%%%%%%%%%%%%%%%%%%%
\subsection{Estimates for $\POD$}\label{subsec:PODEstimates}
\par
 All the hidden constants in this subsection depend only on $\rD$ and $k$.
 Besides the obvious stability estimate
\begin{equation}\label{eq:PODStability1}
  |\POD\zeta|_{\HOD}\leq |\zeta|_{\HOD}\qquad\forall\,\zeta\in \HOD
\end{equation}
 that follows from \eqref{eq:POD1}, we also have
 a stability estimate  for $\POD$  in terms of $\|\cdot\|_\LTD$ and the semi-norm
 $\tbar\cdot\tbar_{k,D}$.
\begin{lemma}\label{lem:PODStability2}
 We have
\begin{equation*}%\label{eq:PODStability2}
 \|\POD\zeta\|_\LTD\lesssim \tbar \zeta\tbar_{k,D} \qquad\forall\,\zeta\in \HOD.
\end{equation*}
\end{lemma}
\begin{proof}
   It follows from \eqref{eq:POD1} that
\begin{align*}
   &\int_D \nabla(\POD \zeta)\cdot\nabla(\POD\zeta)\,dx
   =\int_{\p D}\zeta n\cdot\nabla(\POD\zeta)ds
   -\int_D \zeta \Delta (\POD\zeta)dx\\
  &\hspace{40pt}\leq \Big(\sum_{e\in\ED}\|\Pi_{k-1,e}^0\zeta\|_{L_2(e)}^2\Big)^\frac12
    \Big(\sum_{e\in\ED}\|\nabla\POD\zeta\|_{L_2(e)}^2\Big)^\frac12\\
     &\hspace{80pt}+\|\PDZmt\zeta\|_\LTD\|\Delta(\POD\zeta)\|_\LTD,
\end{align*}
 and we have, by \eqref{eq:DiscreteEstimate0} and \eqref{eq:DiscreteEstimate1},
\begin{align*}
  \sum_{e\in\ED}\|\nabla\POD\zeta\|_{L_2(e)}^2\lesssim
         \hD^{-1}\|\nabla \POD\zeta\|_\LTD^2 \quad\text{and}\quad
  \|\Delta(\POD\zeta)\|_\LTD&\lesssim \hD^{-1}\|\nabla \POD\zeta\|_\LTD.
\end{align*}
 It follows that
\begin{equation*}
  \|\nabla\POD\zeta\|_\LTD
  \lesssim \hD^{-1}\Big(\hD\sum_{e\in\ED}\|\Pi_{k-1,e}^0\zeta\|_{L_2(e)}^2
  +\|\PDZmt\zeta\|_\LTD^2\Big)^\frac12
    =\hD^{-1}\tbar\zeta\tbar_{k,D}.
\end{equation*}
\par
 Moreover \eqref{eq:G2}, \eqref{eq:POD2} and  \eqref{eq:tbarNormDef}, imply
\begin{align*}
    \Big|\int_{\p D} \POD \zeta\,ds\Big|
  =\Big|\sum_{e\in\ED}\int_e \Pi_{0,e}^0\zeta\,ds\Big|
  &\leq
  \sum_{e\in\ED}\sqrt{\hE}\|\Pi_{k-1,e}^0\zeta\|_{L_2(e)}\\
  &\lesssim \sqrt{\hD}
    \Big(\sum_{e\in\ED}\|\Pi_{k-1,e}^0\zeta\|_{L_2(e)}^2\Big)^\frac12\leq
    \tbar\zeta\tbar_{k,D}.
\end{align*}
\par
 Finally we have, by \eqref{eq:PF2}, %and \eqref{eq:tbarNormDef},
\begin{align*}
  \|\POD\zeta\|_\LTD&\leq \Big|\int_{\p D} \POD \zeta\,ds\Big| +
      \hD \|\nabla\POD\zeta\|_\LTD \lesssim \tbar \zeta\tbar_{k,D}.
\end{align*}
\end{proof}
\par
 We can now establish error estimates for $\POD$.
\par
\begin{lemma}\label{lem:PODErrors}
 We have
\begin{alignat}{3}
   \|\zeta-\POD \zeta\|_\LTD&\lesssim \hD^{\ell+1}|\zeta|_{H^{\ell+1}(D)}
 &\qquad&\forall\,\zeta\in H^{\ell+1}(D),\,0\leq\ell\leq k,
 \label{eq:LTwoPOD}\\
   |\zeta-\POD\zeta|_{\HOD}&\lesssim \hD^{\ell}|\zeta|_{H^{\ell+1}(D)}
 &\qquad&\forall\,\zeta\in H^{\ell+1}(D),\,1\leq\ell\leq k,
 \label{eq:HOnePOD}\\
 |\zeta-\POD \zeta|_{H^2(D)}&\lesssim h_D^{\ell-1}|\zeta|_{H^{\ell+1}(D)}&\qquad&
  \forall\,\zeta\in H^{\ell+1}(D),\,1\leq\ell\leq k.\label{eq:HTwoPOD}
\end{alignat}
%
% where the hidden constants depend only on $\rD$ and $k$.
\end{lemma}
\begin{proof}
  The estimate \eqref{eq:HOnePOD} follows immediately from \eqref{eq:BHEstimates},
  \eqref{eq:Projection} and
  \eqref{eq:PODStability1}.
\par
In view of \eqref{eq:DiscreteEstimate1} and \eqref{eq:PODStability1}, we have
\begin{equation*}
   |\POD\zeta|_{H^2(D)}\lesssim \hD^{-1}|\POD\zeta|_{\HOD}\leq\hD^{-1}|\zeta|_{\HOD},
\end{equation*}
 which together with \eqref{eq:BHEstimates} and
  \eqref{eq:Projection} implies \eqref{eq:HTwoPOD}.
\par
 Similarly we have, by \eqref{eq:tbarBdd}
  and Lemma~\ref{lem:PODStability2}
\begin{equation*}
  \|\POD\zeta\|_\LTD\lesssim
  \tbar \zeta\tbar_{k,D}
  \lesssim \|\zeta\|_\LTD+\hD|\zeta|_{\HOD},
\end{equation*}
 which together with \eqref{eq:BHEstimates} and \eqref{eq:Projection} implies
 \eqref{eq:LTwoPOD}.
\end{proof}
%
%%%%%%%%%%%%%%%%%%%%%%%%%%%%%%%%%
\subsection{Estimates for $\PDZ$}\label{subsec:PDZEstimates}
% \par
 All the hidden constants in this subsection only depend on $\rD$ and $k$.
 We have an obvious stability estimate
\begin{equation}\label{eq:PDZLTwo}
   \|\PDZ \zeta\|_\LTD\leq \|\zeta\|_\LTD \qquad\forall\,\zeta\in L_2(D)
 \end{equation}
  and an obvious relation
\begin{equation}\label{eq:PDZInvariance}
 \PDZ q=q\qquad\forall\,q\in \PkD.
\end{equation}
 \par
  It follows from \eqref{eq:BHEstimates}, \eqref{eq:PDZLTwo} and
  \eqref{eq:PDZInvariance} that
\begin{equation} \label{eq:PDZLTwoError}
  \|\zeta-\PDZ\zeta\|_\LTD\lesssim
   \hD^{\ell+1}|\zeta|_{H^{\ell+1}(D)} \qquad\forall\,\zeta\in H^{\ell+1}(D),\,
      0\leq\ell\leq k.
\end{equation}
\par
 We also have a  stability estimate for $\PDZ$ in $|\cdot|_{\HOD}$.
\begin{lemma}\label{lem:PDZHOne}
   We have
\begin{equation}\label{eq:PDZHOne}
     |\PDZ \zeta|_{\HOD}\lesssim |\zeta|_{\HOD} \qquad\forall\,\zeta\in \HOD.
\end{equation}
\end{lemma}
\begin{proof}  This is a consequence of  \eqref{eq:DiscreteEstimate1},
 \eqref{eq:PODStability1},
 \eqref{eq:LTwoPOD} and \eqref{eq:PDZLTwoError}:
\begin{align*}
|\PDZ \zeta|_{\HOD}&\leq |\PDZ \zeta-\POD\zeta|_{\HOD}+|\POD\zeta|_{\HOD}\\
    &\lesssim  \hD^{-1}\|\PDZ \zeta-\POD\zeta\|_\LTD+|\zeta|_{\HOD}\\
 &\lesssim \hD^{-1}\big(\|\PDZ \zeta-\zeta\|_\LTD+\|\zeta-\POD\zeta\|_\LTD\big)
  +|\zeta|_{\HOD}  \lesssim |\zeta|_{\HOD}.
\end{align*}
\end{proof}
\par
 We can then derive error estimates for $\POD$ by combining the
 Bramble-Hilbert estimates
 and Lemma~\ref{lem:PDZHOne}.
\begin{lemma}\label{lem:PDZErrors}
 We have
\begin{alignat}{3}
       |\zeta-\PDZ\zeta|_{\HOD}&\lesssim \hD^\ell|\zeta|_{H^{\ell+1}(D)}
  &\qquad&\forall\,\zeta\in H^{\ell+1}(D), \, 1\leq\ell\leq k, \label{eq:PDZHOneError}\\
  |\zeta-\PDZ\zeta|_\HTD&\lesssim h^{\ell-1}|\zeta|_{H^{\ell+1}(D)}
     &\qquad&\forall\,\zeta\in H^{\ell+1}(D), \, 1\leq\ell\leq k.\label{eq:PDZHTwoError}
\end{alignat}
\end{lemma}
\begin{proof} In view of \eqref{eq:PDZInvariance},
 the estimate \eqref{eq:PDZHOneError} follows from
 \eqref{eq:BHEstimates} and \eqref{eq:PDZHOne}.
\par
Similarly the estimate  \eqref{eq:PDZHTwoError} follows from \eqref{eq:BHEstimates},
 \eqref{eq:PDZInvariance} and the inequality
\begin{equation*}
 |\PDZ\zeta|_{H^2(D)}\lesssim \hD^{-1}|\PDZ\zeta|_\HOD
 \lesssim \hD^{-1}|\zeta|_\HOD
\end{equation*}
 obtained from \eqref{eq:DiscreteEstimate1} and \eqref{eq:PDZHOne}.
\end{proof}
 The following is another useful estimate.
\begin{lemma}\label{lem:PDZcQD}
  We have
\begin{equation*}
  \|\PDZ v\|_\LTD\lesssim C \tbar v\tbar_{k,D} \qquad\forall\,v\in \cQD.
\end{equation*}
\end{lemma}
\begin{proof}    Let $v\in\cQD$ be arbitrary.  It follows from
\eqref{eq:Condition3} that
\begin{align*}
   \|\PDZ v\|_\LTD^2&=\|\PDZmt v\|_\LTD^2+\|(\PDZ-\PDZmt)v\|_\LTD^2\\
     &=\|\PDZmt v\|_\LTD^2+\|(\PDZ-\PDZmt)\POD v\|_\LTD^2\\
      &\leq \|\PDZmt v\|_\LTD^2+\|\POD v\|_\LTD^2,
\end{align*}
 which together with \eqref{eq:tbarNormDef} and Lemma~\ref{lem:PODStability2}
 completes the proof.
\end{proof}
%
%%%%%%%%%%%%%%%%%%%%%%%%%%%%%%%%%%
\subsection{Inverse Estimates}\label{subsec:InverseEstimate}
 These are estimates that bound the norm $|v|_\HOD$ of a virtual element function
 $v\in\cQD$ in terms of $\|\PDZmt v\|_\LTD$ and norms that only involve
 the boundary data of $v$.
 They are crucial for the stability analysis of virtual element methods
in Section~\ref{subsec:DiscreteProblem}.
\par
 We begin with a key lemma.
\begin{lemma}\label{lem:Fundamental}
 There exists a positive constant $C$ depending only on $\rD$ and $k$, such that
\begin{equation}\label{eq:Fundamental}
  |v|_{\HOD}\leq C\big[ \hD^{-1}\tbar v\tbar_{k,D}+|v|_{H^{1/2}(\p D)}\big].
\end{equation}
\end{lemma}
\begin{proof}  By scaling we may assume  $\hD=1$.
\par
 Let $\TL$ be the lifting operator from Section~\ref{subsec:Lifting}.
 The function $w=\TL v\in H^1(D)$ satisfies $w=v$ on $\p D$ and
\begin{equation}\label{eq:w}
   \|w\|_{H^1(D)}\lesssim |v|_{H^{1/2}(D)}.
\end{equation}
\par
 Let $\phi$ be the same (bump) function in the proof of Lemma~\ref{lem:MaximumPrinciple}
  and  $\zeta=w+p\phi$, where the polynomial $p\in \PkD$ is determined by
\begin{equation*}
   \int_D (\zeta-v)q\,dx=0 \qquad\forall\,q\in \PkD,
\end{equation*}
 or equivalently
\begin{equation}\label{eq:Fundamental2}
  \int_D  pq\phi\,dx=\int_D (v-w)q\,dx=\int_D (\PDZ v-w)q\,dx \qquad\forall\,q\in \PkD.
\end{equation}
 Then we have
\begin{equation}\label{eq:Fundamental3}
 |v|_{\HOD}\leq |\zeta|_{\HOD}
\end{equation}
 by Lemma~\ref{lem:MEP} and, in view of \eqref{eq:DiscreteEstimate1},
 \eqref{eq:BumpFunction} and \eqref{eq:w},
\begin{equation}\label{eq:Fundamental4}
 |\zeta|_{\HOD}\leq |w|_{\HOD}+|p\phi|_{\HOD}
   \lesssim |w|_{\HOD}+\|p\|_\LTD\lesssim
   |v|_{H^{1/2}(\p D)}+\|p\|_\LTD.
\end{equation}
\par
 Note that  \eqref{eq:Fundamental1} and \eqref{eq:Fundamental2} imply
\begin{equation*}
 \|p\|_\LTD\lesssim  \|\PDZ v-w\|_\LTD
\end{equation*}
 and hence
\begin{equation}\label{eq:Fundamental5}
  \|p\|_\LTD \lesssim \|\PDZ v\|_\LTD+\|w\|_\LTD\lesssim
   \tbar v\tbar_{k,D}+|v|_{H^{1/2}(\p D)}
\end{equation}
 by Lemma~\ref{lem:PDZcQD} and \eqref{eq:w}.
\par
 The estimate \eqref{eq:Fundamental} (with $\hD=1$) follows from
 \eqref{eq:Fundamental3}--\eqref{eq:Fundamental5}.
\end{proof}
\begin{lemma}\label{lem:InverseEstimate1}
 There exists a positive constant $C$, depending only $\rD$ and $k$, such that
\begin{equation}\label{eq:InverseEstimate1}
  |v|_{\HOD}\leq C \big[\hD^{-1}\tbar v\tbar_{k,D}
  +\hD^{1/2}\|\p v/\p s\|_{L_2(\p D)}\big]
  \qquad\forall\,v\in\cQD,
\end{equation}
 where $\p v/\p s$ is a tangential derivative of $v$.
\end{lemma}
\begin{proof}
 The estimate \eqref{eq:InverseEstimate1}
 follows immediately from
 \eqref{eq:HalfAndOne} and  Lemma~\ref{lem:Fundamental} .
\end{proof}
\begin{lemma}\label{lem:InverseEstimate2}
  There exists a positive constant $C$, depending only on $\rD$, $|\ED|$ and $k$, such
  that
\begin{equation}\label{eq:InverseEstimate2}
 |v|_{\HOD}\leq C \big[ \hD^{-1}\tbar v\tbar_{k,D}
  +\sqrt{\ln(1+\ \tD)}\|v\|_{L_\infty(\p D)} \big]\qquad\forall\,v\in \cQD,
\end{equation}
 where
\begin{equation}\label{eq:TauD}
 \tD= \frac{\max_{e\in\ED}\hE}{\min_{e\in\ED}\hE}.
\end{equation}
\end{lemma}
\begin{proof}
 According to \cite[Lemma~5.1]{BLR:2016:VEM}, we have
\begin{equation}\label{eq:HalfAndInfty}
 |v|_{H^{1/2}(\p D)}\leq C \sqrt{\ln(1+\tD)}\|v\|_{L_\infty(\p D)}
 \qquad\forall\,v\in\cQD,
\end{equation}
 where the positive constant $C$ only depends  on $\rD$, $|\ED|$ and $k$.
 \par
 The estimate \eqref{eq:InverseEstimate2} follows from
  Lemma~\ref{lem:Fundamental}
 and \eqref{eq:HalfAndInfty}.
\end{proof}
\par
 Combining \eqref{eq:tbarBdd2} and \eqref{eq:InverseEstimate2},
 we have the following corollary.
\begin{corollary}\label{cor:InverseEstimate3}
   There exists a positive constant $C$, depending only on
   $\rD$, $|\ED|$ and $k$, such
 that
\begin{equation*}
 |v|_\HOD\leq C\big[ \hD^{-1}\|\PDZmt v\|_\LTD+\sqrt{\factor}
\|v\|_{L_\infty(\p D)}\big]
\qquad\forall\,v\in\cQD.
\end{equation*}
\end{corollary}
%
%%%%%%%%%%%%%%%%%%%%%%%%%%%%%%%%%%
\subsection{The Interpolation Operator}\label{subsec:Interpolation}
 For $s>1$
 the interpolation operator $I_{k,D}:H^s(D)\longrightarrow\cQD$
 is defined by the
 condition that $\zeta$ and $\ID\zeta$ share the same degrees
 of freedom, i.e.,
 $\ID\zeta$ agrees with $\zeta$ at the nodes in $\NPD$ and
\begin{equation}\label{eq:IDDef}
 \PDZmt\ID\zeta=\PDZmt\zeta.
\end{equation}
\par
  It is clear that
\begin{equation}\label{eq:IDInvariance}
  \ID q=q\qquad\forall\,q\in\PkD,
\end{equation}
 and by a standard estimate for polynomials in one variable,
\begin{equation}\label{eq:TrivialBdd}
  \|\ID \zeta\|_{L_\infty(\p D)}\leq C \max_{p\in\NPD}|\zeta(p)|
  \leq\|\zeta\|_{L_\infty(\p D)}\qquad
  \forall\,\zeta\in H^s(D) \;\text{and}\; s>1,
 \end{equation}
 where the positive constant $C$ only depends on $k$.
\par
  For the three dimensional Poisson problem, if the solution belongs to
 $H^{\ell+1}(\O)$, then its restriction to a face $F$
 of a polyhedral subdomain belongs to $H^{\ell+\frac12}(F)$.  Therefore
 below we also consider the interpolants of functions in $H^{\ell+\frac12}(D)$.
\par
 We begin with several stability estimates for the interpolation operator.
\begin{lemma}\label{lem:IDtbar}
  There exists a positive constant $C$, depending only on $\rD$ and $k$, such that
\begin{alignat}{3}
   \tbar\ID\zeta\tbar_{k,D}&\leq C\big[\|\zeta\|_\LTD+\hD|\zeta|_\HOD
   +\hD^2|\zeta|_\HTD\big]
   &\qquad&\forall\,\zeta\in H^2(D),\label{eq:IDtbar1}\\
    \tbar\ID\zeta\tbar_{k,D}&\leq C\big[\|\zeta\|_\LTD
    +\hD|\zeta|_\HOD+\hD^{3/2}|\zeta|_{H^{3/2}(D)}\big]
   &\qquad&\forall\,\zeta\in H^{3/2}(D).\label{eq:IDtbar2}
\end{alignat}
\end{lemma}
\begin{proof} Let $\zeta\in \HTD$ (resp., $H^{3/2}(D)$)
 be arbitrary.  From \eqref{eq:G2},
  \eqref{eq:tbarBdd2}, \eqref{eq:IDDef} and \eqref{eq:TrivialBdd}, we have
\begin{align*}%\label{eq:SpecialIDtbar}
  \tbar\ID\zeta\tbar_{k,D}&\lesssim \hD\|\ID\zeta\|_{L_\infty(\p D)}
  +\|\PDZmt\ID\zeta\|_\LTD\\
  &\lesssim \hD\|\zeta\|_{L_\infty(\p D)}+\|\PDZmt \zeta\|_\LTD\lesssim
  \hD\|\zeta\|_\LinfD,\notag
\end{align*}
 which together with \eqref{eq:Sobolev} (resp., \eqref{eq:Sobolev2})
 implies \eqref{eq:IDtbar1} (resp., \eqref{eq:IDtbar2}).
\end{proof}
\begin{lemma}\label{lem:IDFundamental1}
 We have
\begin{align}
 |\ID\zeta|_{\HOD}&\lesssim |\zeta|_{\HOD}+\hD|\zeta|_{H^2(D)}
 \label{eq:IDHOne}\\
\intertext{for all $\zeta\in\HTD$, and}
  |\ID\zeta|_{\HOD}&\lesssim
   |\zeta|_{\HOD}+\hD^{1/2}|\zeta|_{H^{3/2}(D)}\label{eq:IDHOneHalf}
\end{align}
 for all $\zeta\in H^{3/2}(D)$,
  where the hidden constants only depend on $\rD$, $|\ED|$ and $k$.
\end{lemma}
\begin{proof}  Let $\zeta\in H^2(D)$ be arbitrary and $\bar\zeta_D$
be the mean of $\zeta$ over $D$.
 Since $\ID\bar\zeta_D=\bar\zeta_D$,
  it follows from \eqref{eq:InverseEstimate1} that
\begin{align*}
  |\ID\zeta|_{\HOD}^2&=|\ID(\zeta-\bar\zeta_D)|_\HOD^2\\
  &\lesssim \hD^{-2}\tbar\ID(\zeta-\bar\zeta_D)\tbar_{k,D}^2
  +\hD\|\p (\ID\zeta)/\p s\|_{L_2(\p D)}^2.
\end{align*}
 We have, by a standard interpolation  estimate
 in one dimension and \eqref{eq:Trace} (applied to the first order
 derivatives of $\zeta$),
\begin{align*}
 \hD\|\p(\ID \zeta)/\p s\|_{L_2(\p D)}^2&\lesssim\sum_{e\in\ED}
 \hD\|\p\zeta/\p s\|_{L_2(e)}^2\\
 &\lesssim \sum_{e\in\ED}\big[|\zeta|_\HOD^2+
 \hD^2|\zeta|_\HTD^2\big]\lesssim
 |\zeta|_\HOD^2+\hD^2|\zeta|_\HTD^2.
\end{align*}
\par
 These two estimates together with \eqref{eq:PF2} and \eqref{eq:IDtbar1}
 imply \eqref{eq:IDHOne}.
\par
 Similarly we obtain \eqref{eq:IDHOneHalf} by replacing \eqref{eq:Trace} with
 the estimate in Lemma~\ref{lem:CZ1}.
\end{proof}
\begin{lemma}\label{lem:IDFundamental2}
 We have
\begin{align}
 \|\ID\zeta\|_\LTD&\lesssim  \|\zeta\|_\LTD
 +\hD|\zeta|_{\HOD}+\hD^2|\zeta|_{H^2(D)}
 \label{eq:IDLTwo}\\
\intertext{for all $\zeta\in \HTD$, and}
 \|\ID\zeta\|_\LTD&\lesssim
  \|\zeta\|_\LTD+\hD|\zeta|_{\HOD}+\hD^{3/2}|\zeta|_{H^{3/2}(D)}
  \label{eq:IDLTwoHalf}
\end{align}
 for all $\zeta\in H^{3/2}(D)$,  where the hidden constants only
 depend on $\rD$, $|\ED|$ and $k$.
\end{lemma}
\begin{proof}
  From Lemma~\ref{lem:PODStability2} and \eqref{eq:LTwoPOD} we have
\begin{align}\label{eq:LTwoID}
 \|\ID\zeta\|_\LTD&\lesssim \|\ID\zeta-\POD\ID\zeta\|_\LTD
 +\|\POD\ID\zeta\|_\LTD\\
 &\lesssim \hD|\ID\zeta|_\HOD+\tbar\ID\zeta\tbar_{k,D}.\notag
\end{align}
 The estimates \eqref{eq:IDLTwo} and \eqref{eq:IDLTwoHalf} follow from
 \eqref{eq:LTwoID}, Lemma~\ref{lem:IDtbar}
 and Lemma~\ref{lem:IDFundamental1}.
\end{proof}
\par
 We can now derive error estimates for the interpolation operator.
\begin{lemma}\label{lem:InterpolationError}
 We have, for $1\leq \ell\leq k$,
\begin{alignat}{3}
 \|\zeta-\ID\zeta\|_\LTD+\|\zeta-\POD\ID\zeta\|_\LTD
 &\lesssim \hD^{\ell+1}|\zeta|_{H^{\ell+1}(D)}&\qquad&
 \forall\,\zeta\in H^{\ell+1}(D),
 \label{eq:LTwoPODID}\\
 |\zeta-\ID\zeta|_{\HOD}+|\zeta-\POD\ID\zeta|_{\HOD}
 &\lesssim \hD^{\ell}|\zeta|_{H^{\ell+1}(D)}&\qquad&
 \forall\,\zeta\in H^{\ell+1}(D),
 \label{eq:HOnePODID}\\
 |\zeta-\POD\ID\zeta|_\HTD&\lesssim \hD^{\ell-1}|\zeta|_{H^{\ell+1}(D)}
 &\qquad&\forall\,\zeta\in H^{\ell+1}(D),
  \label{eq:HTwoPODID}
\end{alignat}
 and
\begin{alignat}{3}
 \|\zeta-\ID \zeta\|_\LTD&\lesssim
  \hD^{\ell+\frac12}|\zeta|_{H^{\ell+\frac12}(D)}
  &\qquad&\forall\,\zeta\in H^{\ell+\frac12}(D),
  \label{eq:HalfIDLTwoError}\\
   |\zeta-\ID \zeta|_\HOD&\lesssim
  \hD^{\ell-\frac12}|\zeta|_{H^{\ell+\frac12}(D)}
  &\qquad&\forall\,\zeta\in H^{\ell+\frac12}(D),
  \label{eq:HalfIDHOneError}
\end{alignat}
 where the hidden constants  only depend on $\rD$, $|\ED|$ and $k$.
\end{lemma}
\begin{proof}
   In view of  \eqref{eq:tbarBdd2},
   Lemma~\ref{lem:PODStability2}, \eqref{eq:IDtbar1} and \eqref{eq:IDLTwo},
   we have,
   for any $\zeta\in H^2(D)$,
\begin{align*}
   \|\ID\zeta\|_\LTD+\|\POD \ID\zeta\|_\LTD
   &\lesssim \|\ID\zeta\|_\LTD+\tbar\ID\zeta\tbar_{k,D}\\
     &\lesssim \|\zeta\|_\LTD+\hD|\zeta|_{\HOD}+\hD^2|\zeta|_{H^2(D)},
\end{align*}
 which together with \eqref{eq:BHEstimates},
 \eqref{eq:Projection} and \eqref{eq:IDInvariance} implies
 \eqref{eq:LTwoPODID}.
\par
 From \eqref{eq:PODStability1} and \eqref{eq:IDHOne} we also have,
 for any $\zeta\in H^2(D)$,
\begin{align*}
 & |\ID\zeta|_{H^1(\O)}+|\POD\ID\zeta|_{\HOD}\leq 2|\ID\zeta|_{\HOD}
  \lesssim \hD^{-1}\|\zeta\|_\LTD+|\zeta|_{\HOD}+\hD|\zeta|_{H^2(D)},
\end{align*}
 which together with \eqref{eq:BHEstimates},
 \eqref{eq:Projection} and \eqref{eq:IDInvariance} implies
 \eqref{eq:HOnePODID}.
\par
 Similarly, by using the relation
\begin{align*}
 |\POD\ID\zeta|_\HTD&=|\POD\ID\zeta-\POOD\zeta|_\HTD\\
   &\lesssim \hD^{-1}|\POD\ID\zeta-\POOD\zeta|_\HOD\\
   &\lesssim \hD^{-1}|\ID\zeta|_\HOD+\hD^{-1}|\zeta|_\HOD
   \lesssim \hD^{-2}\|\zeta\|_\LTD+\hD^{-1}|\zeta|_\HOD+|\zeta|_\HTD
\end{align*}
 that follows from \eqref{eq:DiscreteEstimate1},
 \eqref{eq:PODStability1} and
 \eqref{eq:IDHOne}, we can establish \eqref{eq:HTwoPODID}
 through \eqref{eq:BHEstimates}, \eqref{eq:Projection}
 and \eqref{eq:IDInvariance}.
\par
 Finally we obtain \eqref{eq:HalfIDLTwoError} and
 \eqref{eq:HalfIDHOneError}
 by replacing \eqref{eq:BHEstimates}
 (resp., \eqref{eq:IDtbar1}, \eqref{eq:IDHOne} and  \eqref{eq:IDLTwo})
 with \eqref{eq:BHEstimates2}
  (resp., \eqref{eq:IDtbar2}, \eqref{eq:IDHOneHalf} and
  \eqref{eq:IDLTwoHalf})
  in the arguments for \eqref{eq:LTwoPODID} and \eqref{eq:HOnePODID}.
\end{proof}
\par
 The proof for the following result is similar.
\begin{lemma}\label{lem:LTwoIDSpecial}
  There exists a positive constant $C$, depending only on $\rD$, $|\ED|$
  and $k$, such that
 \begin{equation*}%\label{eq:LTwoID}
   \|\ID \zeta-\Pi_{1,D}^0\ID \zeta\|_\LTD\leq C \hD^2|\zeta|_{H^2(D)}
   \qquad\forall\,\zeta\in H^2(D).
 \end{equation*}
\end{lemma}
\par
 We also have interpolation error estimates in the $L_\infty$ norm.
\begin{lemma}\label{lem:1DInterpolationError}
   There exists a positive constant $C$, depending only on $\rD$, $N$ and $k$,
   such that
\begin{alignat}{3}
  \|\zeta-\ID\zeta\|_{L_\infty(D)}&\leq
  C \hD^\ell|\zeta|_{H^{\ell+1}(D)}&\qquad&
  \forall\,\zeta\in H^{\ell+1}(D) \;\text{and}\;
  1\leq\ell\leq k,\label{eq:1DInterpolationError}\\
    \|\zeta-\ID\zeta\|_{L_\infty(D)}&\leq C\hD^{\ell-\frac12}|\zeta|_{H^{\ell+\frac12}(D)}
    &\qquad&
  \forall\,\zeta\in H^{\ell+\frac12}(D) \;\text{and}\;1\leq\ell\leq k.
  \label{eq:1DInterpolationErrorHalf}
\end{alignat}
\end{lemma}
\begin{proof}  It follows from \eqref{eq:Sobolev}, Lemma~\ref{lem:MaximumPrinciple},
 \eqref{eq:TrivialBdd} and \eqref{eq:IDHOne} that, for any $\zeta\in \HTD$,
\begin{equation*}
  \|\ID\zeta\|_\LinfD\lesssim \|\ID\zeta\|_{L_\infty(\p D)}+|\ID\zeta|_\HOD
          \lesssim\hD^{-1}\|\zeta\|_\LTD+|\zeta|_\HOD+\hD|\zeta|_\HTD,
\end{equation*}
 which together with \eqref{eq:BHEstimates} and \eqref{eq:IDInvariance} imply
 \eqref{eq:1DInterpolationError}.
\par
 The proof for \eqref{eq:1DInterpolationErrorHalf} is similar, but with \eqref{eq:Sobolev}
 (resp., \eqref{eq:BHEstimates} and \eqref{eq:IDHOne}) replaced by \eqref{eq:Sobolev2}
 (resp., \eqref{eq:BHEstimates2} and \eqref{eq:IDHOneHalf}).
\end{proof}
%
%%%%%%%%%%%%%%%%%%%%%%%%%%%%%%%%%%
\subsection{The Null Space of $\POD$}\label{subsec:Kernel}
  Let $\KerD=\{v\in\cQD:\,\POD v=0\}$ be the null space of the projection $\POD$.
  The inverse estimates in Section~\ref{subsec:InverseEstimate} can be simplified
  for functions in $\KerD$.
\begin{lemma}\label{lem:KerPOD}
  There exists a positive constant $C$, depending only on $\rD$
  and $k$, such that
\begin{equation}\label{eq:KerPOD}
   \tbar v\tbar_{k,D}^2\leq C
   \hD\sum_{e\in\ED}\|\Pi_{k-1,e}^0 v\|_{L_2(e)}^2
   \qquad\forall\,v\in \KerD.
\end{equation}
\end{lemma}
\begin{proof}  We can assume $k\geq2$ since
 $\Pi_{k-2,D}\hspace{1pt}v=0$ for $k=1$.
  \par
 Let $v\in\KerD$ be arbitrary.  It follows from
 \eqref{eq:POD1} that
\begin{equation}\label{eq:KerPOD1}
  \int_D v(\Delta q)\,dx=\int_{\p D}v(n\cdot\nabla q)\,ds
  \qquad\forall \,q\in \PkD.
\end{equation}
\par
 According to \eqref{eq:discs} and Lemma~\ref{lem:Polynomial},
 given any $p\in\bbP_{k-2}$,
  there exists $q\in\Pk$ such that $\Delta q=p$ and
\begin{equation}\label{eq:KerPOD12}
  \|\nabla q\|_\LTD\leq\|\nabla q\|_{L_2(\tBD)}\lesssim \hD\|p\|_{L_2(\tBD)}
  \lesssim \hD\|p\|_{L_2(\fBD)}\leq \hD\|p\|_\LTD.
\end{equation}
\par
 It follows from
 \eqref{eq:DiscreteEstimate0}, \eqref{eq:KerPOD1} and \eqref{eq:KerPOD12} that
\begin{align*}
    \int_D vp\,dx
 &\leq \Big(\sum_{e\in\ED}\|\Pi_{k-1,e}^0 v\|_{L_2(e)}^2\Big)^\frac12
  \Big(\sum_{e\in\ED}\|\nabla q\|_{L_2(e)}^2\Big)^\frac12\\
 &\lesssim \Big(\sum_{e\in\ED}\|\Pi_{k-1,e}^0 v\|_{L_2(e)}^2\Big)^\frac12 \hD^{-1/2}
    \|\nabla q\|_\LTD\\
 &\lesssim \Big(\sum_{e\in\ED}\|\Pi_{k-1,e}^0 v\|_{L_2(e)}^2\Big)^\frac12 \hD^{1/2}
      \|p\|_\LTD
\end{align*}
  and hence
\begin{equation}\label{eq:KerPod13}
  \|\Pi_{k-2,D}^0v\|_\LTD=\max_{p\in\bbP_{k-2}}\int_D vp\,dx /\|p\|_\LTD
   \lesssim \Big(\hD\sum_{e\in\ED}\|\Pi_{k-1,e}^0 v\|_{L_2(e)}^2\Big)^\frac12.
\end{equation}
\par
 The estimate \eqref{eq:KerPOD} follows from \eqref{eq:tbarNormDef}
 and  \eqref{eq:KerPod13}.
\end{proof}
\begin{lemma}\label{lem:BdryEst}   For any $v\in\PkbD$ that vanishes at
 some point on $\p D$,
   we have
\begin{equation*}
     \|v\|_{L_2(\p D)} \leq C  \hD
            \|\p v/\p s\|_{L_2(\p D)},
\end{equation*}
 where $\p v /\p s$ denotes a tangential derivative of $v$ along $\p D$ and the
 positive constant $C$ only depends on $k$.
\end{lemma}
\begin{proof}  It follows from a Poincar\'e-Friedrichs inequality on
the circle $\p\fBD$
 that
\begin{equation*}
  \|\zeta\|_{L_2(\p\fBD)}\lesssim \hD|\zeta|_{H^1(\p\fBD)}
\end{equation*}
 for any $\zeta\in H^1(\p\fBD)$ that vanishes at some point on $\p\fBD$.
 The lemma then follows from \eqref{eq:L2Iso1} and \eqref{eq:H1Iso1}.
\end{proof}
\par
 Note that every $v\in \KerD$ must vanish at some point on $\p\O$ because
 $\int_{\p D}v\,ds=0$ by \eqref{eq:POD2}.
 Therefore, in view of Lemma~\ref{lem:KerPOD} and Lemma~\ref{lem:BdryEst},
 the inverse estimates
 \eqref{eq:InverseEstimate1} and \eqref{eq:InverseEstimate2} can be simplified to
\begin{alignat*}{3}
  |v|_{\HOD}&\lesssim \hD^{1/2}\|\p v/\p s\|_{L_2(\p D)}
  &\qquad&\forall\,v\in \KerD\\
\intertext{with a hidden constant depending only on $\rD$ and $k$, and}
 |v|_{\HOD}&\lesssim \sqrt{\ln(1+\tD)}\|v\|_{L_\infty(\p D)}
 &\qquad&\forall\,v\in \KerD
\end{alignat*}
 with a hidden constant that also depends on $|\ED|$.
\par
 Hence we have
\begin{alignat}{3}
 |v|_\HOD^2&=|\POD v|_\HOD^2+|v-\POD v|_\HOD^2\label{eq:EnergyBdd1}\\
  &\lesssim |\POD v|_{H^1(D)}^2+\hD\|\p (v-\POD v)/\p s\|_{L_2(\p D)}^2
 &\qquad&\forall\,v\in\cQD,\notag\\
 \intertext{with a hidden constant depending only on $\rD$ and $k$, and also}
 |v|_\HOD^2&\lesssim |\POD v|_{H^1(D)}^2+\ln(1+\tD)\|v-\POD v\|_{L_\infty(\p D)}^2
 &\qquad&\forall\,v\in\cQD,\label{eq:EnergyBdd2}
\end{alignat}
 with a hidden constant that also depends on $|\ED|$.
%
%%%%%%%%%%%%%%%%%%%%%%%%%%%%%%%%%%%%%%%%%%%%%%%%%%%%%
%
\section{The Poisson Problem in Two Dimensions}\label{sec:Poisson2D}
 In this section
 we consider virtual element methods for the Poisson problem
 \eqref{eq:Poisson} in two dimensions
 and establish error estimates under two global shape regularity assumptions.
\par
 Let $\cT_h$ be a triangulation of the convex polygon $\O\subset\R^2$
 by polygonal subdomains,
 where $h=\d\max_{D\in\cT_h}\hD$ is the mesh parameter.
 The global virtual finite element space $\cQh$ is given by
 $$\cQh=\{v\in H^1_0(\O):\, v\big|_D\in\cQD \quad \forall\,D\in\cT_h\}.$$
 The space of (discontinuous) piecewise
  polynomials of degree $\leq k$ with respect to $\cT_h$ is denoted by $\cPh$.
\par
 The operators $\PO :H^1(\O)\longrightarrow \cPh$,
 $\PZ:L_2(\O)\longrightarrow \cPh$
 and $\Ih:H^2(\O)\cap H^1_0(\O)\longrightarrow \cQh$ are defined in terms of
 their local counterparts, i.e.,
 $$(\PO\zeta)\big|_D=\POD(\zeta\big|_D), \quad
 (\PZ\zeta)\big|_D=\PDZ(\zeta\big|_D) \quad
   \text{and} \quad (\Ih\zeta)\big|_D=\ID(\zeta\big|_D).
 $$
\par
 The piecewise $H^1$ norm with respect to $\cT_h$ is given by
\begin{equation}\label{eq:PiecewiseHOneNOrm}
 |v|_{h,1}=\Big(\SumD |v|_{\HOD}^2\Big)^\frac12.
\end{equation}
%
%%%%%%%%%%%%%%%%%%%%%%%%%%%%%%%%%%%%%%%%
\subsection{Global Shape Regularity Assumptions}
\label{subsec:GlobalShapeRegularity}
 We assume that the local shape regularity assumption \eqref{eq:SA}
 is satisfied by all
  $D\in\cT_h$ and impose the following global regularity assumptions.
\par\noindent
{\em Assumption 1}\quad There exists a positive number $\rho\in (0,1)$,
 independent of $h$, such that
\begin{equation}\label{eq:rho}
  \rD\geq\rho \qquad\forall\,D\in\cT_h.
\end{equation}
\par\noindent
{\em Assumption 2}\quad There exists a positive integer $N$,
 independent of $h$, such that
\begin{equation}\label{eq:N}
 |\ED|\leq N\qquad\forall\,D\in\cT_h.
\end{equation}
\par
 The hidden constants in the rest of Section~\ref{sec:Poisson2D}
 will only depend on $\rho$, $N$ and $k$.
%%%%%%%%%%%%%%%%%%%%%%%%%%%%%%%%%%%%%%%%%
\subsection{The Discrete Problem}\label{subsec:DiscreteProblem}
 Let the local stabilizing bilinear form $S^D_1(\cdot,\cdot)$ and
 $S^D_2(\cdot,\cdot)$ be defined by
\begin{align}
 S^D_1(w,v)&=\sum_{\NPD} w(p)v(p),\label{eq:SD1}\\
  S^D_2(w,v)&=\hD(\p w/\p s,\p v/\p s)_{L_2(\p D)},\label{eq:SD2}
\end{align}
 where $\p v/\p s$ denotes a tangential derivative of $v$ along $\p D$.
\begin{remark}\label{rem:SD}{\rm
 The local stability bilinear form $S^D_1(\cdot,\cdot)$ is the boundary part of
 the local stability bilinear form in \cite{BBCMMR:2013:VEM}.
 The bilinear form $S^D_2(\cdot,\cdot)$ was introduced in \cite{WRR:2016:VEM}.}
\end{remark}
\begin{remark}\label{rem:AlternativeSD2}{\rm
  We can also use the bilinear form $\tilde S^D_2(\cdot,\cdot)$ defined by
\begin{equation*}
   \tilde S^D_2(w,v)=\sum_{e\in\ED}\hE(\p w/\p s,\p v/\p s)_{L_2(e)}
\end{equation*}
}
\end{remark}
\begin{remark}\label{rem:NormEquivalence}
{\rm
  By the equivalence of norms on finite dimensional vector spaces, we have
 $$\sum_{p\in\cN_e}v^2(p)\approx \sum_{e\in\ED}\|v\|_{L_\infty(e)}^2
   \approx \|v\|_{L_\infty(\p D)}^2 \qquad \forall\, v\in \PkbD.$$
}
\end{remark}
\par
 The discrete problem for \eqref{eq:Poisson} is to
 find $u_h\in\cQh$ such that
\begin{equation}\label{eq:DiscreteProblem}
   a_h(u_h,v)=(f,\Xi_h v) \qquad \forall\,v\in\cQh,
\end{equation}
 where
\begin{align}
  a_h(w,v)&=\SumD \big[a^D(\POD w,\POD v)
  +S^D(w-\POD w,v-\POD v)\big],\label{eq:ah}\\
  a^D(w,v)&=\int_D \nabla w\cdot\nabla v\,dx,\label{eq:aD}
\end{align}
 $S^D(\cdot,\cdot)$ is either $S^D_1(\cdot,\cdot)$ or
 $S^D_2(\cdot,\cdot)$,
 and $\Xi_h:\cQh\longrightarrow\cPh$ is given by
\begin{equation}\label{eq:Xih}
    \Xi_h=\begin{cases} \Pi_{1,h}^0&\qquad\text{if $k=1,2,$}\\[4pt]
                                 \PZmt &\qquad\text{if $k\geq 3$}.\\
      \end{cases}
\end{equation}
\par
 It follows from \eqref{eq:EnergyBdd1} and \eqref{eq:EnergyBdd2} that
\begin{equation}\label{eq:Stability}
  |v|_{H^1(\O)}^2 \lesssim
    \alpha_h a_h(v,v) \qquad\forall\,v\in \cQh,
\end{equation}
 where
\begin{equation}\label{eq:kappah}
  \alpha_h=\begin{cases}\d
          \ln(1+\max_{D\in\cT_h}\tD)&\qquad
          \text{if $S^D(\cdot,\cdot)=S^D_1(\cdot,\cdot)$}\\
     1&\qquad\text{if $S^D(\cdot,\cdot)=S^D_2(\cdot,\cdot)$}
\end{cases}\;.
\end{equation}
 The well-posedness of the discrete problem follows from the
 stability estimate \eqref{eq:Stability}.
\par
 We will use the following properties of $\Xi_h$ in the error analysis.
\begin{lemma}\label{lem:Xih}
 We have, for $1\leq\ell\leq k$,
\begin{alignat}{3}
  (f,w-\Xi_h w)&\lesssim h^\ell|f|_{H^{\ell-1}(\O)} |w|_{H^1(\O)}
  &\qquad&\forall\,f\in H^{\ell-1}(\O),\;w\in\cQh, \label{eq:Xih1}\\
    (f,\Ih\zeta-\Xi_h\Ih\zeta)&\lesssim
    h^{\ell+1}|f|_{H^{\ell-1}(\O)}|\zeta|_{H^2(\O)}&\qquad&
   \forall f\in H^{\ell-1}(\O),\;\zeta\in H^2(\O).\label{eq:Xih2}
\end{alignat}
\end{lemma}
\begin{proof}
In view of the relation
\begin{align*}
  (f,w-\Xi_h w)&=(f-\PZmt f,w-\Xi_h w)\\
            &\leq \|f-\PZmt f\|_{L_2(\O)}\| w-\Xi_h w\|_{L_2(\O)}
           \leq \|f-\Pi_{\ell-2,h}^0 f\|_{L_2(\O)}
           \|w-\Pi_{0,h}^0 w\|_{L_2(\O)},
\end{align*}
 the estimate \eqref{eq:Xih1} follows from \eqref{eq:PDZLTwoError}.
\par
  Similarly we have
\begin{align*}
   (f,\Ih\zeta-\Xi_h\Ih\zeta)&=(f-\PZmt f,\Ih\zeta-\Xi_h\Ih\zeta)\\
    &\leq \|f-\PZmt f\|_{L_2(\O)}\|\Ih\zeta-\Xi_h\Ih\zeta\|_{L_2(\O)}\\
    &\leq \|f-\Pi_{\ell-2,h}^0f\|_{L_2(\O)}
    \|\Ih\zeta-\Pi_{1,h}^0\Ih\zeta\|_{L_2(\O)},
\end{align*}
 and the estimate \eqref{eq:Xih2} follows from \eqref{eq:PDZLTwoError} and
 Lemma~\ref{lem:LTwoIDSpecial}.
  Note that this is the reason why $\Xi_h$ is chosen to be
 $\Pi_{1,h}^0$ for $k=2$ instead of $\Pi_{k-2,h}^0=\Pi_{0,h}^0$.
\end{proof}
%%%%%%%%%%%%%%%%%%%%%%%%%%%%%%%%
\subsection{An Abstract Error Estimate in the Energy Norm}
\label{subsec:AbstractEnergyError}
 Let $\|\cdot\|_h=\sqrt{a_h(\cdot,\cdot)}$ be the mesh-dependent energy norm.
 Note that \eqref{eq:Stability} implies
\begin{equation}\label{eq:HOneVEM}
 |v|_{H^1(\O)}\lesssim
    \sqrt{\alpha_h}\|v\|_h \qquad\forall\,v\in\cQh.
\end{equation}
\par
 The discrete problem \eqref{eq:DiscreteProblem} is defined in
 terms of a non-inherited
 symmetric positive definite bilinear form.
 We have a standard error estimate (cf. \cite[Lemma~10.1.7]{BScott:2008:FEM}
 and \cite{BSS:1972:NC})
\begin{equation}\label{eq:Standard}
 \|u-u_h\|_h\leq \inf_{v\in\cQh}\|u-v\|_h+\sup_{v\in \cQh}
 \frac{a_h(u,v)-(f,\Xi_h v)}{\|v\|_h}.
\end{equation}
 The key is to control the numerator on the
 right-hand side of \eqref{eq:Standard}.
\par
 In view of \eqref{eq:Poisson}, \eqref{eq:aDef} and \eqref{eq:POD1} we
 can write, for any $v\in\cQh$,
\begin{align*}
  a_h(u,v)&=\SumD \big[a^D(\POD u,\POD v)+S^D(u-\POD u,v-\POD v)\big]\\
  &=\SumD \big[a^D(\POD u,v)+S^D(u-\POD u,v-\POD v)\big]\\
  &=\SumD \big[a^D(\POD u-u,v)+S^D(u-\POD u,v-\POD v)\big]+(f,v)\\
   &=\SumD \big[a^D(\POD u-u,v-\POD v)+S^D(u-\POD u,v-\POD v)\big]+(f,v),
\end{align*}
 and hence, by \eqref{eq:PODStability1}, \eqref{eq:ah} and \eqref{eq:HOneVEM},
\begin{align}\label{eq:Numerator}
 a_h(u,v)-(f,\Xi_h v)&=\SumD \big[a^D(\POD u-u,v-\POD v)
 +S^D(u-\POD u,v-\POD v)\big]\notag\\
 &\hspace{30pt}+(f,v-\Xi_h v),\notag\\
 &\lesssim \Big(\SumD|\POD u-u|_{\HOD}|v-\POD v|_{\HOD}\Big)
 +\|u-\PO u\|_h\|v\|_h\\
  &\hspace{30pt}+\Big(\sup_{w\in\cQh}\frac{(f,w-\Xi_h w)}{|w|_{H^1(\O)}}\Big)
   \sqrt{\alpha_h}\|v\|_h\notag\\
   &\lesssim \|u-\PO u\|_h\|v\|_h+
  \Big(|u-\PO u|_{h,1}+\sup_{w\in\cQh}
  \frac{(f,w-\Xi_h w)}{|w|_{H^1(\O)}}\Big)\sqrt{\alpha_h}\|v\|_h.\notag
\end{align}
\par
 Putting \eqref{eq:Standard} and \eqref{eq:Numerator} together
 we arrive at the
 estimate
\begin{equation}\label{eq:EnergyError}
 \|u-u_h\|_h\lesssim \|u-\Ih u\|_h+\|u-\PO u\|_h
 +\sqrt{\alpha_h}\Big(|u-\PO u|_{h,1}+\sup_{w\in\cQh}
 \frac{(f,w-\Xi_h w)}{|w|_{H^1(\O)}}\Big).
\end{equation}
\par
 Below we will derive concrete error estimates under the assumption
 that the solution $u$ of \eqref{eq:Poisson}  belongs to $H^{\ell+1}(\O)$
 for $1\leq\ell\leq k$.
%
%%%%%%%%%%%%%%%%%%%%%
\subsection{Concrete Error Estimates in the Energy Norm}
\label{subsec:ConcreteEnergyErrors}
 The terms on the right-hand side of \eqref{eq:EnergyError} are
 estimated as follows.
%%%%%%%%%%%%%%%%%
\par\medskip\noindent{\em Estimate for the Term Involving $f$}
\par\smallskip
 Since $u\in H^{\ell+1}(\O)$ and $f=-\Delta u$, we have, by \eqref{eq:Xih1},
\begin{equation}\label{eq:RHSError}
 \sup_{w\in\cQh}\frac{(f,w-\Xi_h w)}{|w|_{H^1(\O)}}
 \lesssim h^{\ell}|u|_{H^{\ell+1}(\O)}.
\end{equation}
%
%%%%%%%%%%%%%%%%%%%%%%%%%%
\par\medskip\noindent{\em Estimate for $|u-\PO u|_{h,1}$}
\par\smallskip
 It follows directly from \eqref{eq:HOnePOD} that
\begin{equation}\label{eq:PODError}
 |u-\PO u|_{h,1}=\Big(\SumD |u-\POD u|_\HOD^2\Big)^\frac12
 \lesssim h^\ell|u|_{H^{\ell+1}(\O)}.
\end{equation}
%
%%%%%%%%%%%%%%%%%%%%%%%%%%
\par\medskip\noindent{\em Estimate for $\|u-\Ih u\|_h$}
\par\smallskip
 We will establish the estimate
\begin{equation}\label{eq:InterpolationError}
 \|u-\Ih u\|_h\lesssim h^{\ell}\unorm
\end{equation}
 for both choices of $S^D(\cdot,\cdot)$.
\par
 In the case where $S^D(\cdot,\cdot)=S^D_1(\cdot,\cdot)$, it follows from
 \eqref{eq:PODStability1},
 \eqref{eq:SD1} and \eqref{eq:ah} that
\begin{align}\label{eq:InterError1}
 \|u-\Ih u\|_h^2&\lesssim \SumD |\POD(u-\ID u)|_\HOD^2+\SumD
  \|(u-\ID u)-\POD(u-\ID u)\|_{L_\infty(\p D)}^2\notag\\
 &\lesssim \SumD \big(|u-\ID u|_\HOD^2+\|u-\ID u\|_{L_\infty(\p D)}^2\big)\\
   &\hspace{40pt}+\SumD |\POD(u-\ID u)|_{L_\infty(D)}^2,\notag
\end{align}
 and we have
\begin{align}\label{eq:InterError2}
 \SumD |\POD(u-\ID u)|_{L_\infty(D)}^2&\lesssim
 \SumD\big(\,(\,\overline{\POD(u-\ID u)}\,)_{\p D}^2
 +|\POD(u-\ID u)|_\HOD^2\big)\notag\\
 &\leq
 \SumD\big(\,(\,\overline{u-\ID u}\,)_{\p D}^2+|u-\ID u|_\HOD^2\big)\\
 &\lesssim \SumD \big(\|u-\ID u\|_{L_\infty(\p D)}^2
 +|u-\ID u|_\HOD^2\big),\notag
\end{align}
 by \eqref{eq:DiscreteEstimate2}, \eqref{eq:Condition3} and
 \eqref{eq:PODStability1}.
\par
 The estimate \eqref{eq:InterpolationError} now follows from
 \eqref{eq:HOnePODID}, \eqref{eq:1DInterpolationError},
  \eqref{eq:InterError1} and \eqref{eq:InterError2}.
\par
 In the case where $S^D(\cdot,\cdot)=S^D_2(\cdot,\cdot)$, it follows from
 \eqref{eq:SD2} and \eqref{eq:ah} that
\begin{align*}
  \|u-\Ih u\|_h^2&\lesssim \SumD|\POD(u-\ID u)|_\HOD^2
     + \SumD\hD \|\p [\POD(u-\ID u)]/\p s\|_{L_2(\p D)}^2\notag\\
       &\hspace{30pt}+{\SumD\hD \|\p(u-\ID u)/\p s\|_{L_2(\p D)}^2}.
\end{align*}
 We have
\begin{equation*}
 \SumD|\POD(u-\ID u)|_\HOD^2\leq \SumD|u-\ID u|_\HOD^2\lesssim
 h^{2\ell}\unorm^2
\end{equation*}
 by \eqref{eq:PODStability1} and \eqref{eq:HOnePODID},
 and hence, in view of
 \eqref{eq:DiscreteEstimate0}  (applied to the first order derivatives
 of the polynomial
 $\POD(u-\ID u)$),
\begin{align*}
 \SumD\hD \|\p [\POD(u-\ID u)]/\p s\|_{L_2(\p D)}^2&\lesssim
    \SumD \|\POD(u-\ID u)|_\HOD^2\lesssim h^{2\ell}\unorm^2.
\end{align*}
 Finally it follows from
 a standard interpolation error estimate in one variable
 and \eqref{eq:HalfTrace} (applied to the $\ell$-th order
 derivatives of $u$) that
\begin{align*}
  \SumD\hD \|\p(u-\ID u)/\p s\|_{L_2(\p D)}^2&\lesssim
  \SumD\hD\sum_{e\in\ED}\hE^{2\ell-1}
        \,|\p^\ell u/\p s^\ell|_{H^{1/2}(e)}^2
        \lesssim h^{2\ell}\unorm^2.
\end{align*}
\par
 Together these estimates imply \eqref{eq:InterpolationError}.
\goodbreak
%%%%%%%%%%%%%%%%%%%%%%%%%%
\par\medskip\noindent{\em Estimate for $\|u-\PO u\|_h$}
\par\smallskip
 We will show that the estimate
\begin{equation}\label{eq:ProjectionError}
 \|u-\PO u\|_h\lesssim h^\ell\unorm
\end{equation}
 holds for both choices of $S^D(\cdot,\cdot)$.
  \par
 In the case where $S^D(\cdot,\cdot)=S^D_1(\cdot,\cdot)$, it follows from
 \eqref{eq:Sobolev}, Lemma~\ref{lem:PODErrors}, \eqref{eq:SD1} and \eqref{eq:ah}
 that
\begin{align*}\label{eq:ProjectionEst2}
 \|u-\PO u\|_h^2
  &\lesssim \SumD \|u-\PO u\|_{L_\infty(\p D)}^2\\
     &\lesssim \SumD\big[\hD^{-2}\|u-\POD u\|_\LTD^2+|u-\POD u|_{\HOD}^2
         +\hD^2|u-\POD u|_{H^2(D)}^2\big]\notag\\
             &\lesssim h^{2\ell}|u|_{H^{\ell+1}(\O)}^2.\notag
\end{align*}
\par
 In the case where $S^D(\cdot,\cdot)=S^D_2(\cdot,\cdot)$,  we have
\begin{align*}
   \|u-\PO u\|_h^2&=\SumD\hD \|\p(u-\POD u)/\p s\|_\LTD^2\\
 &\lesssim \SumD \big[|u-\POD u|_{\HOD}^2+\hD^2|u-\POD u|_{H^2(D)}^2\big]
  \lesssim h^{2\ell}|u|_{H^{\ell+1}(\O)}^2\notag
\end{align*}
 by \eqref{eq:Trace} (applied to the first order derivatives of
  $u-\POD u$) and Lemma~\ref{lem:PODErrors}.
\par
  Putting \eqref{eq:EnergyError}--\eqref{eq:InterpolationError}
  and \eqref{eq:ProjectionError} together, we
 arrive at the following result.
\begin{theorem}\label{thm:ConcreteEnergyError}
  Assuming the solution $u$ of \eqref{eq:Poisson} belongs to
 $H^{\ell+1}(\O)$ for $\ell$ between $1$ and $k$, we have
\begin{equation*}
  \|u-u_h\|_h\leq C\sqrt{\alpha_h} h^\ell|u|_{H^{\ell+1}(\O)},
\end{equation*}
 where $\alpha_h$ is defined in \eqref{eq:kappah} and the positive
 constant $C$ only depends
on $\rho$, $k$ and $N$.
\end{theorem}
\par We have similar estimates for the computable approximate
 solutions $\PO u_h$ and $\PZ u_h$.
\begin{theorem}\label{thm:ComputableEnergyError}
  Assuming the solution $u$ of \eqref{eq:Poisson} belongs
  to $H^{\ell+1}(\O)$
   for some $\ell$ between $1$ and $k$, we have
\begin{equation}\label{eq:ComputableEnergyError}
  |u-u_h|_{H^1(\O)}+\sqrt{\alpha_h}\big[|u-\PO u_h|_{h,1}
  +|u-\PZ u|_{h,1}\big]
 \leq C \alpha_h h^\ell |u|_{H^{\ell+1}(\O)},
\end{equation}
 where $\alpha_h$ is defined in \eqref{eq:kappah} and the positive
 constant $C$ only depends on $\rho$, $N$ and $k$.
\end{theorem}
\begin{proof} In view of \eqref{eq:ah} and Theorem~\ref{thm:ConcreteEnergyError},
 we have
\begin{equation*}
 |\PO (u-u_h)|_{h,1}\leq \|u-u_h\|_h\lesssim
 \sqrt{\alpha_h}h^{\ell}|u|_{H^{\ell+1}(\O)},
\end{equation*}
 which together with  \eqref{eq:PODError}
 implies
\begin{equation*}
 |u-\PO u_h|_{h,1}\leq |u-\PO u|_{h,1}+|\PO(u-u_h)|_{h,1}
 \lesssim \sqrt{\alpha_h}h^{\ell}|u|_{H^{\ell+1}(\O)}.
\end{equation*}
 It follows from this estimate and
   \eqref{eq:PDZHOne} that
\begin{equation*}
  |u-\PZ u|_{h,1}\leq |u-\PO u|_{h,1}+|\PZ(\PO u-u)|_{h,1}
  \lesssim |u-\PO u|_{h,1}\lesssim
  \sqrt{\alpha_h}h^{\ell}\unorm.
\end{equation*}
\par
 Finally we have, by \eqref{eq:HOnePODID}, \eqref{eq:HOneVEM},
 \eqref{eq:InterpolationError}
 and Theorem~\ref{thm:ConcreteEnergyError},
\begin{align*}
 |u-u_h|_{H^1(\O)}&\leq |u-\Ih u|_{H^1(\O)}+|\Ih u-u_h|_{H^1(\O)}\\
 &\lesssim |u-\Ih u|_{H^1(\O)}+\sqrt{\alpha_h}\|\Ih u-u_h\|_h\\
  &\leq |u-\Ih u|_{H^1(\O)}+\sqrt{\alpha_h}
  \big(\|u-\Ih u\|_h+\|u-u_h\|_h\big)
  \lesssim \alpha_h h^{\ell}\unorm.
\end{align*}
\end{proof}
\goodbreak
\begin{remark}\label{rem:HOneErrors}{\rm
 In the case where $S^D(\cdot,\cdot)=S^D_1(\cdot,\cdot)$,
 the estimates for $|u-\PO u_h|_{h,1}$ and $|u-\PZ u_h|_{h,1}$
 are better than the
 estimate for $|u-u_h|_{H^1(\O)}$.
}
\end{remark}
%%%%%%%%%%%%%%%%%%%%%%%
\subsection{Error Estimates in the $L_2$ Norm}\label{subsec:L2Error}
 We begin with two lemmas involving $S^D(\cdot,\cdot)$.
\begin{lemma}\label{lem:SDEstimate}
 We have
\begin{equation}\label{eq:SDEstimate}
 \SumD
 S^D(\zeta-\POD\ID \zeta,\zeta-\POD\ID \zeta)\lesssim
 h^2|\zeta|_{H^2(\O)}^2\qquad\forall\,
  \zeta\in H^2(\O)\cap H^1_0(\O).
\end{equation}
\end{lemma}
\begin{proof}
  It follows from \eqref{eq:Trace} (applied to the first
  order partial derivatives of
  $\zeta-\POD\ID\zeta$), Lemma~\ref{lem:InterpolationError}
  and \eqref{eq:SD2} that
\begin{align*}
   &\SumD
   S^D_2(\zeta-\POD\ID \zeta,\zeta-\POD\ID \zeta)
   =\SumD\hD\|\p(\zeta-\POD\ID \zeta)/\p s\|_{L_2(\p D)}^2\\
   &\hspace{60pt}\lesssim \SumD \big[|\zeta-\POD\ID\zeta|_{\HOD}^2+
         \hD^2|\zeta-\POD\ID\zeta|_{H^2(D)}^2\big]
         \lesssim \hD^2|\zeta|_{H^2(\O)}^2,
\end{align*}
 and we have
\begin{align*}
  &\SumD
  S^D_1(\zeta-\POD\ID\zeta,\zeta-\POD\ID \zeta)\lesssim
 \SumD\|\zeta-\POD\ID \zeta\|_{L_\infty(\p D)}^2\\
  &\hspace{30pt}\lesssim \SumD\big[\hD^{-2}\|\zeta-\POD\ID \zeta\|_\LTD^2
  +|\zeta-\POD\ID \zeta|_{\HOD}^2+\hD^2|\zeta-\POD\ID \zeta|_{H^2(D)}^2\big]\\
  &\hspace{30pt}\lesssim \hD^2|\zeta|_{H^2(\O)}^2
\end{align*}
 by \eqref{eq:Sobolev}, Lemma~\ref{lem:InterpolationError} and
 \eqref{eq:SD1}.
\end{proof}
\begin{lemma}\label{lem:POSD}
    Assuming that $u\in H^{\ell+1}(\O)$ for some $\ell$
    between $1$ and $k$, we have
\begin{equation}\label{eq:POSD}
  \SumD S^D(u_h-\POD u_h,u_h-\POD u_h)\lesssim
  \alpha_h^2 h^{2\ell}\unorm^2.
\end{equation}
\end{lemma}
\begin{proof}
  This is a consequence of \eqref{eq:ah},
\eqref{eq:ProjectionError}, Theorem~\ref{thm:ConcreteEnergyError}
  and Theorem~\ref{thm:ComputableEnergyError}:
\begin{align*}
  &\SumD S^D(u_h-\POD u_h,u_h-\POD u_h)= \|u_h-\PO u_h\|_h^2\\
  &\hspace{100pt}\lesssim \|u_h-u\|_h^2+\|u-\PO u\|_h^2
  +\|\PO(u-u_h)\|_h^2\\
  &\hspace{100pt}=\|u_h-u\|_h^2+\|u-\PO u\|_h^2+|u-u_h|_{H^1(\O)}^2
  \lesssim \alpha_h^2 h^{2\ell}\unorm^2.
\end{align*}
\end{proof}
\par
 We can now prove a consistency estimate.
\begin{lemma}\label{lem:ConsistencyError}
  Assuming that $u\in H^{\ell+1}(\O)$ for some $\ell$
  between $1$ and $k$, we have
\begin{equation*}
  a(u-u_h,\Ih\zeta)\leq C{\alpha_h}
   h^{\ell+1}|u|_{H^{\ell+1}(\O)}
     |\zeta|_{H^2(\O)} \qquad \forall\,
     \zeta\in H^2(\O)\cap H^1_0(\O),
\end{equation*}
 where $\alpha_h$ is defined in \eqref{eq:kappah} and
 the positive constant $C$
 only depends on $\rho$, $k$ and $N$.
\end{lemma}
\begin{proof}
 We have, by \eqref{eq:Poisson}, \eqref{eq:aDef}, \eqref{eq:POD1} and
 \eqref{eq:DiscreteProblem}--\eqref{eq:aD},
\begin{align*}%\label{eq:ConsistencyFormula1}
 &a(u-u_h,\Ih\zeta)=a(u,\Ih\zeta)-\SumD a^D(u_h,\ID\zeta)\notag\\
           &\hspace{40pt}=(f,\Ih\zeta)-\SumD a^D(u_h,\POD\ID\zeta)
           -\SumD a^D(u_h,\ID\zeta-\POD\ID\zeta)\\
           &\hspace{40pt}=(f,\Ih\zeta)-a_h(u_h,\ID\zeta)
           +\SumD S^D(u_h-\POD u_h,\ID\zeta-\POD\ID\zeta)\\
           &\hspace{80pt}+\SumD a^D(\POD u_h-u_h,\ID\zeta
           -\POD\ID\zeta)\\
 &\hspace{40pt}=(f,\Ih\zeta-\Xi_h\Ih\zeta)
 +\SumD S^D(u_h-\POD u_h,\ID\zeta-\POD\ID\zeta)\\
      &\hspace{80pt}+\SumD a^D(\POD u_h-u_h,\ID\zeta-\POD\ID\zeta),
\end{align*}
 and the three terms on the right-hand side  can be estimated as follows.
\par
 We have
\begin{equation*}
 |(f,\Ih\zeta-\Xi_h\Ih\zeta)|\lesssim h^{\ell+1}\unorm|\zeta|_{H^2(\O)}
\end{equation*}
 by \eqref{eq:Xih2},
\begin{equation*}
\sum_{D\in\cT_h} {S^D(u_h-\POD u_h,\ID\zeta-\POD\ID\zeta)}
 \lesssim {\alpha_h}h^{\ell+1}\unorm|\zeta|_{H^2(\O)}
\end{equation*}
 by Lemma~\ref{lem:SDEstimate} and Lemma~\ref{lem:POSD},
 and
\begin{align*}
& \sum_{D\in\cT_h} a^D(\POD u-u_h,\ID\zeta-\POD \ID\zeta)\\
    &\hspace{40pt}\leq \big[|\POD u-u|_{h,1}+|u-u_h|_{H^1(\O)}\big]
   \big[|\ID\zeta-\zeta|_{H^1(\O)}+|\zeta-\POD\ID\zeta|_{h,1}\big]\\
     &\hspace{40pt}\lesssim {\alpha_h}
      h^{\ell+1}\unorm|\zeta|_{H^2(\O)}
\end{align*}
 by Lemma~\ref{lem:InterpolationError}, \eqref{eq:PODError}
 and Theorem~\ref{thm:ComputableEnergyError}.
\end{proof}
\begin{theorem}\label{thm:uhLTwoError}
 Assuming $u\in H^{\ell+1}(\O)$ for some $\ell$ between $1$ and $k$,
 there exists a positive constant $C$, depending only on
 $\rho$, $N$ and $k$, such that
\begin{equation}\label{eq:uhLTwoError}
 \|u-u_h\|_{L_2(\O)}\leq C{\alpha_h}h^{\ell+1}|u|_{H^{\ell+1}(\O)},
\end{equation}
 where $\alpha_h$ is defined in \eqref{eq:kappah}.
\end{theorem}
\begin{proof}
 Let $\zeta\in H^1_0(\O)$ be defined by
\begin{equation*}
   a(v,\zeta)=(v,u-u_h)  \qquad\forall\,v\in H^1_0(\O).
\end{equation*}
 We have
\begin{equation}\label{eq:Duality1}
   \|u-u_h\|_{L_2(\O)}^2=a(u-u_h,\zeta)
                          = a(u-u_h,\zeta-\Ih\zeta)+a(u-u_h,\Ih\zeta),
\end{equation}
 and since $\O$ is convex,
\begin{equation}\label{eq:Duality2}
  \|\zeta\|_{H^2(\O)}\leq C_\O \|u-u_h\|_{L_2(\O)}
\end{equation}
 by elliptic regularity \cite{Grisvard:1985:EPN,Dauge:1988:EBV}.
\par
 The first term on the right-hand side of  \eqref{eq:Duality1} satisfies
\begin{equation}\label{eq:Duality3}
   a(u-u_h,\zeta-\Ih\zeta)\leq |u-u_h|_{H^1(\O)}
   |\zeta-\Ih\zeta|_{H^1(\O)}\\
          \lesssim  h|u-u_h|_{H^1(\O)}|\zeta|_{H^2(\O)}
\end{equation}
 by \eqref{eq:HOnePODID}, and then
 \eqref{eq:uhLTwoError} follows from
 Theorem~\ref{thm:ComputableEnergyError},
 Lemma~\ref{lem:ConsistencyError}
 and \eqref{eq:Duality1}--\eqref{eq:Duality3}.
\end{proof}
\par
 We have similar $L_2$ error estimates for the computable
 approximations  $\PZ u_h$ and $\PO u_h$.
\begin{theorem}\label{thm:ComputableLTwoErrors}
 Assuming $u\in H^{\ell+1}(\O)$ for some $\ell$ between $1$ and $k$,
  there exists a positive constant $C$, depending only on $\rho$,
  $N$ and $k$,  such that
\begin{equation*}
 \|u-\PZ u_h\|_{L_2(\O)}
 +\|u-\PO u_h\|_{L_2(\O)}\leq C{\alpha_h}h^{\ell+1}|u|_{H^{\ell+1}(\O)},
\end{equation*}
 where $\alpha_h$ is defined in \eqref{eq:kappah}.
\end{theorem}
\begin{proof} The estimate for $\PZ u_h$ follows from
 \eqref{eq:PDZLTwoError}, Theorem~\ref{thm:uhLTwoError}
 and the relation
\begin{align*}
  \|u-\PZ u_h\|_{L_2(\O)}&\leq \|u-\PZ u\|_{L_2(\O)}
  +\|\PZ(u-u_h)\|_{L_2(\O)}\\
     &\leq \|u-\PZ u\|_{L_2(\O)}+\|u-u_h\|_{L_2(\O)}.
\end{align*}
\par
 For the approximation $\PO u_h$, we have
\begin{equation}\label{eq:LTwoPODError1}
 \|u-\PO u_h\|_{L_2(\O)}\leq\|u-\PO\Ih u\|_{L_2(\O)}
 +\|\PO(\Ih u-u_h)\|_{L_2(\O)}
\end{equation}
 and, in view of \eqref{eq:Trace}, \eqref{eq:tbarNormDef},
  and Lemma~\ref{lem:PODStability2},
\begin{align}\label{eq:LTwoPODError2}
\|\PO(\Ih u-u_h)\|_{L_2(\O)}^2&\lesssim
\sum_{D\in\cT_h}\tbar \ID u-u_h\tbar_{k,D}^2\notag\\
    &\lesssim \sum_{D\in\cT_h}\big(\hD\|\ID u-u_h\|_{L_2(\p D)}^2
    +\|\PDZmt (\ID u-u_h)\|_\LTD^2\big)\\
    &\lesssim\sum_{D\in\cT_h} \big(\|\ID u-u_h\|_\LTD^2
    +\hD^2|\ID u-u_h|_{\HOD}^2\big)\notag\\
    &\lesssim \|u-u_h\|_{L_2(\O)}^2+ \|u-\Ih u\|_{L_2(\O)}^2
    +h^2|u-\Ih u|_{H^1(\O)}^2\notag\\
       &\hspace{50pt}+h^2|u-u_h|_{H^1(\O)}^2.\notag
\end{align}
\par
 The estimate for $\PO u_h$ follows from
   Lemma~\ref{lem:InterpolationError},
 Theorem~\ref{thm:ComputableEnergyError},
 Theorem~\ref{thm:uhLTwoError}
 and \eqref{eq:LTwoPODError1}--\eqref{eq:LTwoPODError2}.
\end{proof}
%
%%%%%%%%%%%%%%%%%%%%%%%%
\subsection{Error Estimates for $u_h$ in the $L_\infty$ Norm}\label{subsec:uhLInftyError}
 Here we consider a $L_\infty$ error estimate for $u_h$ over the edges of $\cT_h$,
 where $u_h$ is computable.
 We will treat the two choices of $S^D(\cdot,\cdot)$ separately.  The set of all the
 edges in $\cT_h$ will be denoted by $\cE_h$.
\subsubsection{The Case where $S^D(\cdot,\cdot)=S^D_2(\cdot,\cdot)$}
\label{eq:subsubsec:uhLInftySD2}
 We have the following result for this choice of $S^D(\cdot,\cdot)$.
\begin{theorem}\label{thm:MaxNormBdryEst2}
 Assuming that
 the solution $u$ of \eqref{eq:Poisson}
  belongs to $\in H^{\ell+1}(\O)$ for some $\ell$ between $1$ and $k$,
 we have
\begin{equation*}
  \max_{e\in\cE_h}\|u-u_h\|_{L_\infty(e)}\leq C
  h^{\ell}\unorm,
\end{equation*}
 where the positive constant $C$ only depends on $\rho$, $N$ and  $k$.
\end{theorem}
\begin{proof} First we observe that,
  by \eqref{eq:SD2}, \eqref{eq:ah} and Theorem~\ref{thm:ConcreteEnergyError},
\begin{equation}\label{eq:EdgeEstimate}
 \SumD\sum_{e\in\ED}\hD\|\p[(u-u_h)-\POD (u-u_h)]/\p s\|_{L_2(e)}^2
 \lesssim \|u-u_h\|_h^2\lesssim h^{2\ell}\unorm^2.
\end{equation}
\par
 We can connect any point in $e\in\cE_h$
 to $\p\O$, where $u-u_h=0$, by a path along the edges in $\cE_h$.  Therefore it follows from
 a direct calculation (or a Sobolev inequality in one variable)
 that
\begin{align}\label{eq:PathEstimate}
  \|u-u_h\|_{L_\infty(e)}^2&\lesssim \SumD\sum_{e\in\ED}
   \hE\|\p(u-u_h)/\p s\|_{L_2(e)}^2\notag\\
    &\lesssim \SumD\sum_{e\in\ED}\hD\|\p[(u-u_h)-\POD (u-u_h)]/\p s\|_{L_2(e)}^2\\
    &\hspace{40pt}+\SumD\sum_{e\in\ED}
   \hD\|\p[\POD(u-u_h)]/\p s\|_{L_2(e)}^2\qquad\qquad\forall\,e\in\cE_h,\notag
\end{align}
 and we have,
 by \eqref{eq:Trace}, \eqref{eq:DiscreteEstimate1},
 \eqref{eq:PODStability1} and Theorem~\ref{thm:ComputableEnergyError},
\begin{align}\label{eq:TrickyEstimate}
  &\SumD\sum_{e\in\ED} \hD\|\p[\POD(u-u_h)]/\p s\|_{L_2(e)}^2\notag\\
  &\hspace{80pt}\lesssim \SumD \big(|\POD(u-u_h)|_\HOD^2+\hD^2|\POD(u-u_h)|_\HTD^2\big)\\
  &\hspace{80pt}\lesssim\SumD |\POD(u-u_h)|_\HOD^2\lesssim  h^{2\ell}\unorm^2.
  \notag
\end{align}
\par
 The estimates \eqref{eq:EdgeEstimate}--\eqref{eq:TrickyEstimate} together imply
\begin{equation*}
   \|u-u_h\|_{L_\infty(e)}\leq h^\ell\unorm \qquad\forall\,e\in\cE_h.
\end{equation*}
\end{proof}
%
%%%%%%%%%%%%%%%%%%%%%%%%%%%%%%%%%%%%%
\subsubsection{The Case where $S^D(\cdot,\cdot)=S^D_1(\cdot,\cdot)$}
\label{eq:subsubsec:uhLInftySD1}
 We will establish an analog of Theorem~\ref{thm:MaxNormBdryEst2} under the
 additional assumption that
 $\cT_h$ is quasi-uniform, i.e., there exists a positive constant
 $\gamma$  independent of $h$ such that
\begin{equation}\label{eq:gamma}
  \hD\geq \gamma h \qquad\forall\,\gamma\in\cT_h.
\end{equation}
\begin{theorem}\label{thm:MaxNormBdryEst1}
 Assuming $\cT_h$ is quasi-uniform and the solution $u$ of
 \eqref{eq:Poisson} belongs to $H^{\ell+1}(\O)$ for some $\ell$ between $1$ and $k$,
 we have
\begin{equation*}
  \max_{e\in\cE_h}\|u-u_h\|_{L_\infty(e)}\leq
  C\ln(1+\max_{D\in\cT_h}\tD)
  h^{\ell}\unorm,
\end{equation*}
 where the positive constant $C$ only depends on $\rho$,
 $N$, $\gamma$ and  $k$.
\end{theorem}
\begin{proof} Let $D\in\cT_h$ be arbitrary.
 First we observe that,
 by \eqref{eq:SD1}, Remark~\ref{rem:NormEquivalence},
 \eqref{eq:ah}, \eqref{eq:InterpolationError},
 Theorem~\ref{thm:ConcreteEnergyError} and
 Theorem~\ref{thm:ComputableEnergyError},
\begin{align}\label{eq:SD1Edge1}
 &\|(\ID u-u_h)-\POD(\ID u-u_h)\|_{L_\infty(\p D)}\notag\\
 &\hspace{80pt}\lesssim \|\Ih u-u_h\|_h\\
 &\hspace{80pt}
 \lesssim \|\Ih u-u\|_h+\|u-u_h\|_h
  \lesssim [\ln(1+\max_{D\in\cT_h}\tD)]^\frac12h^\ell\unorm.\notag
\end{align}
\par
 Furthermore, it follows from \eqref{eq:G2}, \eqref{eq:DiscreteEstimate3},
 \eqref{eq:PODStability1}, \eqref{eq:LTwoPODID},
  Theorem~\ref{thm:ConcreteEnergyError},
 Theorem~\ref{thm:ComputableLTwoErrors} and \eqref{eq:gamma} that
\begin{align}\label{eq:SD1Edge2}
  &\|\POD(\ID u-u_h)\|_{L_\infty(\p D)}\lesssim
  \hD^{-1}\|\POD(\ID u-u_h)\|_\LTD+
   |\POD(\ID u-u_h)|_\HOD\notag\\
   &\hspace{60pt}\lesssim \hD^{-1}\big(\|\POD\ID u-u\|_\LTD
   +\|u-\POD u_h\|_\LTD\big)+|\ID u-u_h|_\HOD\\
   &\hspace{60pt}\lesssim \ln(1+\max_{D\in\cT_h}\tD) h^\ell\unorm.\notag
\end{align}
\par
 The theorem follows from \eqref{eq:1DInterpolationError},
 \eqref{eq:SD1Edge1}, \eqref{eq:SD1Edge2} and the triangle inequality.
\end{proof}
%
%%%%%%%%%%%%%%%%%%%%%%%%%%%%%%%%%%%%%%%%%%%%%%%%%%%%%%%%%%%%%%%%%
\subsection{Error Estimates for $\PO u_h$ and $\PZ u_h$ in the $L_\infty$ Norm}
\label{subsec:ComputableLInftyError}
\par
 Again we treat the two choices of $S^D(\cdot,\cdot)$ separately.
%%%%%%%%%%%%%%%%%%%%%%%%%%
\subsubsection{The Case where $S^D(\cdot,\cdot)=S^D_2(\cdot,\cdot)$}
\label{eq:subsubsec:ComputableLInftySD2}
 For this choice of $S^D(\cdot,\cdot)$, we can establish the
 following result without
 assuming that $\cT_h$ is quasi-uniform.
\begin{theorem}\label{thm:LInftyErrorsSD2}
  Assuming the solution $u$ of \eqref{eq:Poisson} belongs to
  $H^{\ell+1}(\O)$ for some $\ell$ between $1$ and $k$,
 there exists a positive constant $C$, depending only on $\rho$,
 $N$ and $k$, such that
\begin{equation*}
 \|u-\PO u_h\|_{L_\infty(\O)}+\|u-\PZ u_h\|_{L_\infty(\O)}
  \leq C  h^{\ell}\unorm.
\end{equation*}
\end{theorem}
\begin{proof}  For any $D\in\cT_h$, we have,  by \eqref{eq:Sobolev},
 \eqref{eq:DiscreteEstimate1} and \eqref{eq:PODStability1},
\begin{align*}
 \|u-\POD u_h\|_{\LinfD}
 &\lesssim
 \hD^{-1}\|u-\POD u_h\|_\LTD+|u-\POD u_h|_\HOD+\hD|u-\POD u_h|_\HTD\\
 &\lesssim \hD^{-1}\|u-\POD u_h\|_\LTD+|u-\POD u_h|_\HOD+\hD|u-\POD u|_\HTD\\
             &\hspace{50pt}+  \hD|\POD (u-u_h)|_\HTD\\
  &\lesssim \hD^{-1}\|u-\POD u_h\|_\LTD+|u-\POD u_h|_\HOD+\hD|u-\POD u|_\HTD\\
  &\hspace{40pt}+|u-u_h|_\HOD,
\end{align*}
 and
\begin{align*}
  \hD^{-1}\|u-\POD u_h\|_\LTD&\leq \hD^{-1}\|u-u_h\|_\LTD+
      \hD^{-1}\|u_h-\POD u_h\|_\LTD\\
      &\lesssim \big(\|u-u_h\|_{L_\infty(\p D)}+|u-u_h|_\HOD\big)
        +|u_h-\POD u_h|_\HOD
\end{align*}
 by \eqref{eq:G2}, \eqref{eq:PF2} and \eqref{eq:POD2}.  These two
 estimates together with
 Lemma~\ref{lem:PODErrors},
 Theorem~\ref{thm:ComputableEnergyError} and
 Theorem~\ref{thm:MaxNormBdryEst2} imply
 the estimate for $u-\PO u_h$.
\par
 The estimate for $u-\PZ u_h$ can be derived similarly.
 We have, by \eqref{eq:Sobolev},
 \eqref{eq:DiscreteEstimate1} and
 \eqref{eq:PDZHOne},
\begin{align*}
 \|u-\PDZ u_h\|_{\LinfD}
 &\lesssim
 \hD^{-1}\|u-\PDZ u_h\|_\LTD+|u-\PDZ u_h|_\HOD+\hD|u-\PDZ u|_\HTD\\
 &\hspace{40pt}+|u-u_h|_\HOD,
\end{align*}
 and
\begin{align*}
  \hD^{-1}\|u-\PDZ u_h\|_\LTD&\leq \hD^{-1}\|u-u_h\|_\LTD+
           \hD^{-1}\|u_h-\PDZ u_h\|_\LTD\\
  &\lesssim \big(\|u-u_h\|_{L_\infty(\p D)}+|u-u_h|_\HOD\big)+
  |u_h-\PDZ u_h|_\HOD
\end{align*}
 by \eqref{eq:G2}, \eqref{eq:PF1} and \eqref{eq:PF2}.
 The estimate for $u-\PZ u_h$ now follows from
 Lemma~\ref{lem:PDZErrors},
  Theorem~\ref{thm:ComputableEnergyError} and
 Theorem~\ref{thm:MaxNormBdryEst2}.
\end{proof}
%
%%%%%%%%%%%%%%%%%%%%%%%%%%
\subsubsection{The Case where $S^D(\cdot,\cdot)=S^D_1(\cdot,\cdot)$}
\label{eq:subsubsec:ComputableLInftySD1}
 The following analog of Theorem~\ref{thm:LInftyErrorsSD2} is
 proved by the same arguments but with
 Theorem~\ref{thm:MaxNormBdryEst2} replaced by
 Theorem~\ref{thm:MaxNormBdryEst1}.
\begin{theorem}\label{thm:LInftyErrorsSD1}
  Assuming $\cT_h$ is quasi-uniform and the solution $u$ of \eqref{eq:Poisson}
  belongs to $H^{\ell+1}(\O)$ for some $\ell$ between $1$ and $k$,
 there exists a positive constant $C$, depending only on $\rho$,
  $N$, $\gamma$ and $k$, such that
\begin{equation*}
 \|u-\PO u_h\|_{L_\infty(\O)}+\|u-\PZ u_h\|_{L_\infty(\O)}
  \leq C \ln(1+\max_{D\in\cT_h}\tD)  h^{\ell}\unorm.
\end{equation*}
\end{theorem}
%%%%%%%%%%%%%%%%%%%%%%%%
\section{Virtual Element Methods for the Poisson Problem in
Three Dimensions}\label{sec:Poisson3D}
 The analysis of virtual element methods in three dimensions
 follows the same strategy as in two dimensions
 and many of the results in Section~\ref{sec:LocalVEM2D} and
 Section~\ref{sec:Poisson2D} carry over by identical arguments.
 We will only provide details for estimates that require different derivations.
\par
 Let $\cT_h$ be a polyhedral mesh on $\O$. The set of the faces of a
 subdomain $D\in\cT_h$ is denoted by
 $\FD$ and the set of the edges of
 $F$ is denoted by $\EF$.    The set of all the faces of
 $\cT_h$ is denoted by $\cF_h$ and the set of all the edges of
  $\cT_h$ is denoted by
 $\cE_h$.
\subsection{Shape Regularity Assumptions in Three Dimensions}
\label{subsec:Shape3D}
 We impose the following shape regularity assumptions on $\cT_h$,
 where $\hD$ is the diameter of $D$.
\par\smallskip\noindent
{\em Assumption 1} \quad
 There exists $\rho\in (0,1)$, independent of $h$,
  such that every polyhedron $D\in\cT_h$ is star-shaped with
 respect to a ball $\fBD$ with radius $\geq \rho \hD$.
\par\smallskip\noindent
 {\em Assumption 2} \quad There exists a positive integer $N$,
 independent of $h$, such that
  $|\FD|\leq N$ for all $D\in\cT_h$.
 \par\smallskip\noindent
 {\em Assumption 3}\quad The shape regularity assumptions in
 Section~\ref{subsec:GlobalShapeRegularity}
 are satisfied by all the faces in $\cF_h$, with the same $\rho$ from
 Assumption~1 and the same
 $N$ from Assumption~2.
\par\smallskip
 All the hidden constants below will only depend on $\rho$, $N$ and $k$.
 \par
 Let $D$ be a polyhedron in $\cT_h$.
 We can define the inner product $(\!(\cdot,\cdot)\!)$ by
 \eqref{eq:InnerProduct} where the infinitesimal arc-length $ds$
 is replaced by
 the infinitesimal
 surface area $dS$.  Then the projection operator
 $\POD:H^1(D)\longrightarrow\PkD$ is defined by
 \eqref{eq:PODDef} and
\begin{equation}\label{eq:3DPODStability1}
  |\POD\zeta|_\HOD\leq |\zeta|_\HOD\qquad\forall\,\zeta\in H^1(D).
\end{equation}
 The projection from $\LTD$ to $\PkD$ is again denoted by $\PDZ$.
\par
 The results in Section~\ref{sec:SS}
  are valid for $D\in\cT_h$ under Assumption 1.
  Consequently the results in Section~\ref{subsec:PODEstimates} and
  Section~\ref{subsec:PDZEstimates}  are also valid
  provided the semi-norm $\tbar\cdot\tbar_{k,D}$ is defined by the
  following analog of
  \eqref{eq:tbarNormDef}:
\begin{equation}\label{eq:tbarNorm3D}
  \tbar \zeta\tbar_{k,D}^2=\|\PDZmt \zeta\|_\LTD^2
  +\hD\sum_{F\in\FD}\|\PFZmo \zeta\|_\LTF^2,
\end{equation}
 where $\PFZmo$ is the projection from $L_2(F)$ onto $\bbP_{k-1}(F)$.
\par
 We have
 the following estimates for $\POD$ and  $\PDZ$:
\begin{equation}\label{eq:3DLTwoPODPDZ}
   \|\zeta-\POD \zeta\|_\LTD+ \|\zeta-\PDZ\zeta\|_\LTD
   \lesssim \hD^{\ell+1}|\zeta|_{H^{\ell+1}(D)}
 \quad\forall\,\zeta\in H^{\ell+1}(D),\,0\leq\ell\leq k,
\end{equation}
 and for $1\leq\ell \leq k$,
\begin{alignat}{3}
   |\zeta-\POD\zeta|_{\HOD}+ |\zeta-\PDZ\zeta|_{\HOD}&\lesssim
   \hD^{\ell}|\zeta|_{H^{\ell+1}(D)}
 &\quad&\forall\,\zeta\in H^{\ell+1}(D), \label{eq:3DHOnePODPDZ}\\
  |\zeta-\POD \zeta|_{H^2(D)}+|\zeta-\PDZ\zeta|_\HTD&
  \lesssim h_D^{\ell-1}|\zeta|_{H^{\ell+1}(D)}
 &\quad&
  \forall\,\zeta\in H^{\ell+1}(D).\label{eq:3DHTwoPODPDZ}
\end{alignat}
\par
 The analogs of Lemma~\ref{lem:PODStability2} and
 \eqref{eq:tbarNorm3D} lead to the
 estimate
\begin{equation}\label{eq:3DPODStability2}
  \|\POD \zeta\|_\LTD^2\lesssim \tbar\zeta\tbar_{k,D}^2
  \lesssim \|\zeta\|_\LTD^2+\hD\|\zeta\|_{L_2(\p D)}^2\qquad
  \forall\,\zeta\in H^1(D),
\end{equation}
 and we also have the following analog of \eqref{eq:PDZHOne}:
\begin{equation}\label{eq:3DPDZHOne}
 |\PDZ \zeta|_\HOD\lesssim |\zeta|_\HOD \qquad\forall\,\zeta\in\HOD.
\end{equation}
%
%%%%%%%%%%%%%%%%%%%%%%%%%
\subsection{The Local Virtual Element Space $\bm\cQD$}
\label{subsec:Local3D}
 The space $\cQbD$ of
 continuous piecewise (two dimensional) virtual element functions of order
 $\leq k$ on $\p D$ is defined by
\begin{equation}\label{eq:cQbDDef}
  \cQbD=\{v\in C(\p D):\,v\big|_F\in \cQ^k(F) \quad\forall\,F\in \FD\}.
\end{equation}
\par
 For $k\geq 1$,
 the virtual element space $\cQD\subset \HOD$ is defined by the
 following conditions:  $v\in \HOD$ belongs to $\cQD$ if and only if
 (i) the trace of $v$ on $\p D$ belongs to $\cQbD$,
  (ii) the distribution $-\Delta v$ belongs to $\PkD$,
 and (iii) condition \eqref{eq:Condition3} is satisfied.
\begin{remark}\label{rem:3DContinuity}{\rm
  Since the restriction of $v\in\cQD$ to $\p D$ belongs to
  $C(\p D)$ and $-\Delta v\in \PkD$, the virtual
  element function $v$ is also continuous
 on $\bar D$ (cf. \cite[Section~1.2]{Kenig:1994:CBMS}).}
\end{remark}
\begin{remark}\label{rem:3Ddofs}{\rm
 The degrees of freedom of $\cQD$ (cf. \cite{AABMR:2013:Projector})
 consist of (i) the values of $v$ at the
 vertices of $D$ and nodes on
 the interior of each edge of $D$ that determine a polynomial of degree $k$
 on each edge of $D$,
 (ii) the moments of $\PFZmt v$ on each face $F$ of $D$, and
 (iii) the moments of
 $\PDZmt v$ on $D$.}
\end{remark}
\begin{remark}\label{rem:3DComputable}{\rm
  For $v\in\cQD$ and $F\in\FD$, the polynomial
  $\PFZ v$ can be computed in terms of the degrees of freedom of $v\big|_F$
   (cf. Remark~\ref{rem:Computable}).  Therefore the polynomial
  $\POD v$ can be computed in terms of the degrees of freedom of
  $v\in\cQD$ through \eqref{eq:POD1} and \eqref{eq:POD2}.
  The polynomial $\PDZ v$ can then be computed through
  \eqref{eq:Condition3}.
}
\end{remark}
\begin{remark}\label{rem:DAndF}{\rm
 Under Assumption~3 in Section~\ref{subsec:Shape3D}, the results in
 Section~\ref{sec:LocalVEM2D} (with $D$ replaced by $F$)
 are valid for the restriction of $v\in \cQD$
 to any face $F$ of $D$.
 }
 \end{remark}
\par
 The three dimensional analogs of Lemma~\ref{lem:MEP},
 Lemma~\ref{lem:Fundamental}
 and Lemma~\ref{lem:InverseEstimate1}
 lead to the estimate
\begin{equation}\label{eq:3DInverse}
 |v|_{H^1(D)}^2\lesssim \hD^{-2}\tbar\zeta\tbar_{k,D}^2
      +\hD\sum_{F\in\FD}\|\GF v\|_\LTF^2 \qquad\forall\,v\in\cQD,
\end{equation}
 where $\GF$ is the two dimensional gradient operator on the face $F$,
 and we also have an analog  of \eqref{eq:KerPod13}:
\begin{equation}\label{eq:3DKerPOD}
  \|\PDZmt v\|_\LTD^2\lesssim \hD\SumF \|\PFZmo v\|_\LTF^2
  \qquad\forall\,v\in \KerD.
\end{equation}
 Hence we have, by \eqref{eq:tbarNorm3D}, \eqref{eq:3DInverse} and
 \eqref{eq:3DKerPOD},
\begin{equation}\label{eq:3DInverse2}
  |v|_{H^1(D)}^2\lesssim \hD^{-1}\sum_{F\in\FD}\|\PFZmo v\|_\LTF^2\\
         +\hD\sum_{F\in\FD}\|\GF v\|_\LTF^2\qquad\forall\,v\in\KerD.
\end{equation}
\par
 The interpolation operator $\ID:H^2(D)\longrightarrow \cQD$ is defined by the
 condition that $\ID\zeta$ and $\zeta$ share the same degrees of freedom.
 In particular we have
\begin{equation}\label{eq:3DIDInvariance}
 \ID q=q\qquad\forall\,q\in \PkD.
\end{equation}
\par
 Note that
\begin{equation}\label{eq:ConsistentID}
  \IF (\zeta\big|_F)=(\ID\zeta)\big|_F \qquad\forall\,F\in\FD,
\end{equation}
 and hence, in view of \eqref{eq:G2}, Lemma~\ref{lem:MaximumPrinciple}
 and \eqref{eq:tbarNorm3D},
\begin{align}\label{eq:3DIDtbar}
  \tbar \ID\zeta\tbar_{k,D}^2&=\|\PDZmt (\ID\zeta)\|_\LTD^2
  +\hD\SumF\|\PFZmo (\ID\zeta)\|_\LTF^2\notag\\
  &=\|\PDZmt\zeta\|_\LTD^2+\hD\SumF \|\PFZmo(\IF\zeta)\|_\LTF^2\\
  &\lesssim \|\zeta\|_\LTD^2+\hD\SumF\hF^2\|\IF\zeta\|_{L_\infty(F)}^2\notag\\
  &\lesssim \|\zeta\|_\LTD^2+\hD\SumF\hF^2\big(\|\IF\zeta\|_{L_\infty(\p F)}^2
  +\|\GF(\IF\zeta)\|_\LTF^2\big).\notag
\end{align}
\par
 The error estimates for $\ID$ rely on the following analog of \eqref{eq:IDHOne},
 where $\tF$ is defined by replacing $D$ by $F$ in \eqref{eq:TauD}.
\begin{lemma}\label{lem:3DIDHOne}
 We have
\begin{equation}\label{eq:3DIDHOne}
 |\ID\zeta|_\HOD\lesssim \hD^{-1}\|\zeta\|_\LTD
 +|\zeta|_{H^1(D)}+\hD|\zeta|_\HTD
\end{equation}
 for all $\zeta\in\HTD$.
\end{lemma}
\begin{proof}
 Let $\zeta \in H^2(D)$ be arbitrary.  It follows from
 \eqref{eq:3DInverse}
 and \eqref{eq:3DIDtbar} that
\begin{align*}
  |\ID\zeta|_\HOD^2
  &\lesssim \hD^{-2}\|\zeta\|_\LTD^2+\hD \SumF \|\IF \zeta\|_{L_\infty(\p F)}^2
  +\hD\SumF \|\GF(\IF\zeta)\|_\LTF^2.
\end{align*}
 We have, by \eqref{eq:Sobolev} and \eqref{eq:TrivialBdd},
\begin{align}\label{eq:IFLinfty}
  \hD\SumF\|\IF\zeta\|_{L_\infty(\p F)}^2&\lesssim
  \hD\SumF \|\zeta\|_{L_\infty(\p F)}^2\\
     &\lesssim \hD\|\zeta\|_\LinfD^2\lesssim \hD^{-2}\|\zeta\|_\LTD^2
     +|\zeta|_\HOD^2+\hD^2|\zeta|_\HTD^2,\notag
\end{align}
 and by \eqref{eq:HalfTrace}, Lemma~\ref{lem:CZ2} and \eqref{eq:IDHOneHalf},
\begin{align}\label{eq:IFHone}
  \hD\SumF \|\GF(\IF\zeta)\|_\LTF^2&\lesssim
    \hD\SumF\big(|\zeta|_\HOF^2+\hF |\zeta|_{H^{3/2}(F)}^2\big)\\
   &\lesssim \SumD\big( |\zeta|_\HOD^2+\hD^{2}|\zeta|_\HTD^2\big).
   \notag
%  \factorTD \hD\|\zeta\|_{\LinfD}^2\\
%  &\hspace{60pt}\lesssim \factorTD \big(\hD^{-2}\|\zeta\|_\LTD^2
%  +|\zeta|_\HOD^2+\hD^2|\zeta|_\HTD^2\big).\notag
\end{align}
\end{proof}
\par
 Note that \eqref{eq:3DIDtbar}, \eqref{eq:IFLinfty} and \eqref{eq:IFHone} imply
\begin{equation}\label{eq:3DIDtbar2}
  \tbar\ID\zeta\tbar_{k,D}\lesssim
  \|\zeta\|_\LTD+\hD|\zeta|_\HOD+\hD^2|\zeta|_\HTD \quad\forall\,\zeta\in\HTD,
\end{equation}
 and hence we have, in view of \eqref{eq:3DLTwoPODPDZ},
 \eqref{eq:3DPODStability2}, \eqref{eq:3DIDtbar} and \eqref{eq:3DIDHOne},
\begin{align}\label{eq:3DIDLTwo}
  \|\ID\zeta\|_\LTD&\leq \|\ID\zeta-\POD\ID\zeta\|_\LTD+\|\POD\ID\zeta\|_\LTD\notag\\
    &\lesssim \hD|\ID\zeta|_\HOD+\tbar\ID\zeta\tbar_{k,D}\\
    &\lesssim \big(\|\zeta\|_\LTD
 +\hD|\zeta|_{H^1(D)}+\hD^2|\zeta|_\HTD\big)\notag
\end{align}
 for all $\zeta\in H^2(D)$.
\par
 In view of \eqref{eq:3DIDInvariance},
 the following analogs of \eqref{eq:LTwoPODID}--\eqref{eq:HTwoPODID},
 where $\zeta\in H^{\ell+1}(D)$ and $1\leq\ell\leq k$,
  can be obtained by
 combining the Bramble-Hilbert estimates \eqref{eq:BHEstimates} with the stability
 estimates \eqref{eq:3DPODStability1},
 \eqref{eq:3DPODStability2}, \eqref{eq:3DIDHOne},
 \eqref{eq:3DIDtbar2} and \eqref{eq:3DIDLTwo}.
\begin{alignat}{3}
 \|\zeta-\ID\zeta\|_\LTD+\|\zeta-\POD\ID\zeta\|_\LTD
 &\lesssim \hD^{\ell+1}|\zeta|_{H^{\ell+1}(D)}
     \label{eq:3DLTwoPODID}\\
  |\zeta-\ID\zeta|_{\HOD}+|\zeta-\POD\ID\zeta|_{\HOD}
 &\lesssim \hD^{\ell}|\zeta|_{H^{\ell+1}(D)}
   \label{eq:3DHOnePODID}\\
    |\zeta-\POD\ID\zeta|_\HTD&\lesssim
    \hD^{\ell-1}|\zeta|_{H^{\ell+1}(D)}
  \label{eq:3DHTwoPODID}
\end{alignat}
\begin{remark}\label{rem:3DLInftyInterpolation}{\rm
  We also have the following analog of \eqref{eq:1DInterpolationError}:
\begin{equation*}
  \|\zeta-\ID\zeta\|_\LinfD\lesssim h^{\ell-\frac12}|u|_{H^{\ell+1}(D)}
\end{equation*}
 for all $\zeta\in H^{\ell+1}(D)$ and $1\leq \ell\leq k$.  The proof uses
 Lemma~\ref{lem:MaximumPrinciple} (which is valid in three dimensions) and the
 arguments for \eqref{eq:1DInterpolationError}.  But we do not need this estimate in
 the error analysis.
}
\end{remark}
%%%%%%%%%%%%%%%%%%%%%%%%%%%%
\subsection{The Discrete Problem}\label{subsec:DP3D}
 Let the global virtual element space $\cQh$ be defined by
 $$\cQh=\{v\in H^1_0(\O):\,v\big|_D\in\cQD\quad\forall\,D\in\cT_h\}.$$
 The discrete problem for \eqref{eq:Poisson} is to find $u_h\in\cQh$ such that
\begin{equation*}
  a_h(u_h,v)=(f,\Xi_h v) \qquad\forall\,v\in\cQh,
\end{equation*}
 where $\Xi_h$ is defined as in \eqref{eq:Xih},
\begin{equation*}
  a_h(w,v)=\SumD\big[\big(\nabla(\POD w),\nabla(\POD v)\big)_\LTD+S^D(w-\POD w,v-\POD v)\big],
\end{equation*}
 and the local stabilizing bilinear form $S^D(\cdot,\cdot)$ is given by
\begin{align}
    S^D(w,v)&=\hD\SumF\Big(\hF^{-2}(\PFZmt w,\PFZmt v)_\LTF+
   \sum_{p\in\NPF}w(p)v(p)\Big).\label{eq:3DSD}
\end{align}
 Here $\NPF$ is the set of the nodes along $\p F$ associated with the degrees of freedom
 of a virtual element function.
\begin{lemma}\label{lem:3DSDBdd}
 There exists a positive constant $C$, depending only on $\rho$, $N$ and $k$, such that
\begin{alignat}{3}
 |v|_\HOD^2&\leq C \big [\ln(1+\max_{F\in\FD}\tF)\big] S^D(v,v)
 &\qquad&\forall\,v\in\KerD. \label{eq:3DSDBdd}
\end{alignat}
\end{lemma}
\begin{proof}  Let $v\in\KerD$ be arbitrary.  We have, by  \eqref{eq:G2},
 \eqref{eq:PF2}, Lemma~\ref{lem:MaximumPrinciple}, Corollary~\ref{cor:InverseEstimate3}
  and \eqref{eq:3DInverse2},
\begin{align*}
 |v|_\HOD^2&\lesssim  \hD\SumF \big(\hD^{-2}\|v\|_\LTF^2
 +\|\GF v\|_\LTF^2\big)\notag\\
      &\lesssim \hD\SumF\Big(\hD^{-2}\hF^2\big(\|v\|_{L_\infty(\p F)}^2
      +\|\GF v\|_{L_2(F)}^2\big)+\|\GF v\|_\LTF^2\Big)\\
      &\lesssim \hD\SumF\big(\|v\|_{L_\infty(\p F)}^2+\|\GF v\|_\LTF^2\big)\\
      &\lesssim \hD\SumF\big(\hF^{-2}\|\PFZmt v\|_\LTF^2
      +\ln(1+\tF)\|v\|_{L_\infty(\p F)}^2\big),
\end{align*}
  which together with Remark~\ref{rem:NormEquivalence} and \eqref{eq:3DSD}
  implies \eqref{eq:3DSDBdd}.
\end{proof}
\par
 It follows from \eqref{eq:3DSDBdd} that we have an analog of \eqref{eq:Stability}:
 \begin{equation}\label{eq:3DStability}
  |v|_{H^1(\O)}^2\leq 2\SumD\big[|\PO v|_{H^1(\O)}^2
  +|v-\POD v|_{\HOD}^2\big]\lesssim
    \beta_h a_h(v,v) \qquad\forall\,v\in \cQh,
\end{equation}
 where
\begin{equation}\label{eq:lambdah}
  \beta_h=     \ln(1+\max_{F\in\cF_h}\tF).
\end{equation}
 Hence the discrete problem is well-posed.
\begin{remark}\label{rem:3DInfo}{\rm
 The constants in the error estimates for the virtual element methods
  will only depend on
 $\rho$, $N$, $k$ and $\beta_h$. Therefore the existence of small
 faces in $\cT_h$ does not
 affect the performance of the method.  It is only the relative sizes of
 the edges on each face that matter.
 }
\end{remark}
\par
 Note that the estimates in Lemma~\ref{lem:Xih} are also valid
 for $\O\subset \R^3$.
%%%%%%%%%%%%%%%%%%%%
\subsection{Error Estimates in the Energy Norm}\label{subsec:3DEnergyError}
 The abstract error estimate
\begin{equation}\label{eq:3DAbstractEnergyError}
 \|u-u_h\|_h\lesssim \|u-\Ih u\|_h+\|u-\PO u\|_h
  +\sqrt{\beta_h}\Big(|u-\PO u|_{h,1}+
  \sup_{w\in\cQh}\frac{(f,w-\Xi_hw)}{|w|_{H^1(\O)}}\Big)
\end{equation}
 is obtained by the same arguments as in Section~\ref{subsec:AbstractEnergyError},
 where $|\cdot|_{h,1}$ is defined  in \eqref{eq:PiecewiseHOneNOrm}.
\par
 We will derive concrete error estimates under the assumption that
 $u$ belongs to $H^{\ell+1}(\O)$ for $1\leq\ell\leq k$.
  Since the estimate
\begin{equation}\label{eq:3DEasy}
  |u-\PO u|_{h,1}+ \sup_{w\in\cQh}\frac{(f,w-\Xi_hw)}{|w|_{H^1(\O)}}
  \lesssim h^\ell\unorm
\end{equation}
  remains the same,
  we only need to estimate $\|u-\Ih u\|_h$ and $\|u-\PO u\|_h$.
\par
 It follows from \eqref{eq:G1},  \eqref{eq:3DPODStability1},
 \eqref{eq:ConsistentID} and
 \eqref{eq:3DSD} that
\begin{align}\label{eq:3DInterpolationError1}
  \|u-\Ih u\|_h^2&\lesssim \SumD|\POD(u-\ID u)|_\HOD^2
  +\SumD\hD\SumF \hF^{-2}\|u-\ID u\|_\LTF^2\notag\\
  &\hspace{60pt}+\SumD\hD\SumF\hF^{-2}\|\POD(u-\ID u)\|_\LTF^2\notag\\
    &\hspace{90pt}+\SumD \hD\SumF \|\POD (u-\ID u)\|_{L_\infty(\p F)}^2\\
        &\lesssim \SumD |u-\ID u|_\HOD^2
        +\SumD\hD\SumF \hF^{-2}\|u-\IF u\|_\LTF^2\notag\\
        &\hspace{60pt}+\SumD \hD\|\POD (u-\ID u)\|_{L_\infty(D)}^2,\notag
\end{align}
 and we have, by \eqref{eq:HalfTrace} and \eqref{eq:HalfIDLTwoError},
\begin{equation}\label{eq:3DInterpolationError2}
  \SumD\hD\SumF \hF^{-2}\|u-\IF u\|_\LTF^2
  \lesssim \SumD \hD \SumF \hF^{2\ell-1}|u|_{H^{\ell+(1/2)}(F)}^2
   \lesssim h^{2\ell}\unorm^2.
\end{equation}
 Moreover the estimates \eqref{eq:G1}, \eqref{eq:DiscreteEstimate3} and
 \eqref{eq:3DPODStability1} imply
\begin{align}\label{eq:3DInterpolationError3}
 &\SumD \hD\|\POD (u-\ID u)\|_{L_\infty(D)}^2 \notag\\
 &\hspace{60pt}
  \lesssim   \SumD \hD^{-2}\|\POD (u-\ID u)\|_\LTD^2
  +\SumD|\POD (u-\ID u)|_\HOD^2\\
  &\hspace{60pt}\lesssim
      \SumD \hD^{-2}\big(\|\POD u-u)\|_\LTD^2
      +\|u-\POD\ID u)\|_\LTD^2\big)\notag\\
      &\hspace{100pt}\notag+\SumD|u-\ID u|_\HOD^2.\notag
\end{align}
\par
 Combining \eqref{eq:3DLTwoPODPDZ}, \eqref{eq:3DLTwoPODID},
 \eqref{eq:3DHOnePODID} and
  \eqref{eq:3DInterpolationError1}--\eqref{eq:3DInterpolationError3},
 we obtain
\begin{equation}\label{eq:3DInterpolationError}
        \|u-\Ih u\|_h^2\lesssim  h^{2\ell}\unorm^2,
\end{equation}
 which is the analog of \eqref{eq:InterpolationError}.
\par
 From \eqref{eq:Sobolev}, \eqref{eq:G1},
 \eqref{eq:3DLTwoPODPDZ}--\eqref{eq:3DHTwoPODPDZ}
  and \eqref{eq:3DSD},  we find
\begin{align}\label{eq:3DProjectionError}
  \|u-\PO u\|_h^2&\lesssim %\SumD |u-\POD u|_\HOD^2
  \SumD\hD\SumF\hF^{-2}\|u-\POD u\|_\LTF^2
   +\SumD \hD\SumF \|u-\POD u\|_{L_\infty(\p F)}^2\notag\\
   &\lesssim \SumD \hD\|u-\POD u\|_{L_\infty(D)}^2\\
   &\lesssim    \SumD\big( \hD^{-2}\|u-\POD u\|_\LTD^2
      +|u-\POD u|_\HOD^2+\hD^2|u-\POD u|_\HTD^2\big)\notag\\
      &\lesssim h^{2\ell}\unorm^2,\notag
 \end{align}
 which is the analog of \eqref{eq:ProjectionError}.
\par
 The estimates \eqref{eq:3DAbstractEnergyError}, \eqref{eq:3DEasy},
 \eqref{eq:3DInterpolationError} and \eqref{eq:3DProjectionError}
% together we arrive at
 lead to the following analog of Theorem~\ref{thm:ConcreteEnergyError}.
\begin{theorem}\label{thm:3DConcreteEnergyError}
   Assuming the solution $u$ of \eqref{eq:Poisson} belongs to
 $H^{\ell+1}(\O)$ for $\ell$ between $1$ and $k$, we have
\begin{equation}\label{eq:3DConcreteEnergyError}
  \|u-u_h\|_h\leq C \sqrt{\beta_h}h^\ell\unorm.
\end{equation}
 where $\beta_h$ is defined in \eqref{eq:lambdah} and
 the positive constant $C$ depends only on $\rho$, $N$ and $k$.
\end{theorem}
\par
 The following analog of
 Theorem~\ref{thm:ComputableEnergyError} on the computable approximate
 solutions $\PO  u_h$ and $\PZ u_h$
 is obtained by the same arguments.
 \begin{theorem}\label{thm:3DComputableEnergyError}
    Assuming the solution $u$ of \eqref{eq:Poisson} belongs to
    $H^{\ell+1}(\O)$
   for $\ell$ between $1$ and $k$, there exists a positive constant $C$,
   depending only on
   $\rho$, $N$ and $k$, such that
 \begin{equation}\label{eq:3DComputableEnergyError}
  |u-u_h|_{H^1(\O)}+\sqrt{\beta_h}\big[|u-\PO u_h|_{h,1}
  +|u-\PZ u|_{h,1}\big]\leq
  C \beta_h h^\ell |u|_{H^{\ell+1}(\O)},
\end{equation}
 where $\beta_h$ is defined in \eqref{eq:lambdah}.
\end{theorem}
%
%%%%%%%%%%%%%%%%%%%%%%%%%%%%
\subsection{Error Estimates in the $L_2$ Norm}\label{subsec:3DL2Error}
 We begin with an analog of Lemma~\ref{lem:SDEstimate}.
\begin{lemma}\label{lem:3DSDEstimate}
 We have
\begin{equation*}%\label{eq:3DSDEstimate}
\SumD S^D(\zeta-\PO\Ih \zeta,\zeta-\PO\Ih \zeta)\lesssim
  h^2|\zeta|_{H^2(\O)}^2\qquad\forall
  \zeta\in H^2(\O)\cap H^1_0(\O).
\end{equation*}
\end{lemma}
\begin{proof} It follows from \eqref{eq:Sobolev},
\eqref{eq:3DLTwoPODID}--\eqref{eq:3DHTwoPODID} and
\eqref{eq:3DSD} that
\begin{align*}
 &\SumD S^D(\zeta-\PO \Ih\zeta,\zeta-\PO \Ih\zeta)\\
 &\hspace{30pt}\lesssim
 \SumD\hD\SumF\big(\hF^{-2}\|\zeta-\POD\ID\zeta\|_\LTF^2
   +\|\zeta-\POD\ID\zeta\|_{L_\infty(\p F)}^2\big)\\
   &\hspace{30pt}\lesssim
   \SumD\hD\|\zeta-\POD\ID\zeta\|_{L_\infty(\p D)}^2\\
   &\hspace{30pt}\lesssim
   \SumD \big(\hD^{-2}\|\zeta-\POD\ID\zeta\|_\LTD^2+
   |\zeta-\POD\ID\zeta|_\HOD^2+\hD^2|\zeta-\POD\ID\zeta|_\HTD^2\big)\\
   &\hspace{30pt}\lesssim \hD^2|\zeta|_{H^2(\O)}^2.
\end{align*}
\end{proof}
\goodbreak
\par
  The same arguments as in the proof of Lemma~\ref{lem:POSD} lead
   to the following result.
\begin{lemma}\label{lem:3DPOSD}
  We have
\begin{equation*}
  \SumD S^D(u_h-\PO u_h,u_h-\PO u_h)\lesssim \beta_h^2 h^{2\ell}\unorm^2.
\end{equation*}
\end{lemma}
\par
 With Lemma~\ref{lem:3DSDEstimate} and Lemma~\ref{lem:3DPOSD} in hand,
 we obtain
 the following analog of Theorem~\ref{thm:uhLTwoError}
  and Theorem~\ref{thm:ComputableLTwoErrors}
 by identical arguments.

 \begin{theorem}\label{thm:3DLTwoErrors}
 Assuming $u\in H^{\ell+1}(\O)$ for some $\ell$ between $1$ and $k$,
 there exists a positive constant $C$, depending only on $N$, $k$ and $\rho$,
  such that
\begin{align*}
 \|u-u_h\|_{L_2(\O)}+\|u-\PZ u_h\|_{L_2(\O)}+ \|u-\PO u_h\|_{L_2(\O)}\leq C\beta_h
 h^{\ell+1}|u|_{H^{\ell+1}(\O)},
\end{align*}
 where $\beta_h$ is defined in \eqref{eq:lambdah}.
\end{theorem}
%
%%%%%%%%%%%%%%%%%%%%%%%%%%%%
\subsection{Error Estimate in the $L_\infty$ Norm}\label{subsec:3DInftyError}
 We will derive $L_\infty$ error estimates under the additional
 assumption that $\cT_h$ is quasi-uniform (cf. \eqref{eq:gamma}).
 We begin with an analog of Theorem~\ref{thm:MaxNormBdryEst1}.
\begin{theorem}\label{thm:uhInftyBdryEst}
      Assuming $\cT_h$ is quasi-uniform and the solution $u$ of \eqref{eq:Poisson} belongs to
 $H^{\ell+1}(\O)$ for $\ell$ between $1$ and $k$, we have
\begin{equation}\label{eq:InftyBdryEst}
  \max_{e\in\cE_h}\|u-u_h\|_{L_\infty(e)}\leq C\beta_h h^{\ell}\unorm,
\end{equation}
 where $\beta_h$ is defined in \eqref{eq:lambdah} and
  the positive constant $C$ only depends on $\rho$, $N$, $\gamma$ and $k$.
\end{theorem}
\begin{proof}  It follows from Remark~\ref{rem:NormEquivalence}
 that
\begin{equation*}
  \SumD\hD\SumF\|(\ID u-u_h)-\POD (\ID u-u_h)\|_{L_\infty(\p F)}^2
   \lesssim \|\Ih u-u_h\|_h^2
\end{equation*}
 and hence, for any $D\in\cT_h$ and $F\in\FD$,
\begin{equation}\label{eq:3DBdryEst1}
  \|(\ID u-u_h)-\POD (\ID u-u_h)\|_{L_\infty(\p F)}^2\lesssim
    \beta_h\, h^{2\ell-1}\unorm^2
\end{equation}
 by  \eqref{eq:3DInterpolationError}, Theorem~\ref{thm:3DConcreteEnergyError}
 and the quasi-uniformity of $\cT_h$,
\par
 For any $D\in\cT_h$ and $F\in\FD$,
   we have, by \eqref{eq:G1}, \eqref{eq:DiscreteEstimate3},
   \eqref{eq:3DPODStability1},
   \eqref{eq:3DLTwoPODID}, \eqref{eq:3DHOnePODID},
   Theorem~\ref{thm:3DLTwoErrors}
   and the quasi-uniformity of $\cT_h$,
\begin{align}\label{eq:3DBdryEst2}
  &\|\POD(\ID u-u_h)\|_{L_\infty(F)}^2\notag\\
    &\hspace{40pt}\lesssim \hD^{-3}\|\POD(\ID u-u_h)\|_\LTD^2+
      \hD^{-1}|\POD(\ID u-u_h)|_\HOD^2\notag\\
    &\hspace{40pt}\lesssim \hD^{-3}\big(\|\POD\ID u-u\|_\LTD^2
    +\|u-\POD u_h\|_\LTD^2\big) % \\
    % &\hspace{70pt}
    +\hD^{-1}|\ID u-u_h|_\HOD^2\\
    &\hspace{40pt}\lesssim \beta_h^2 h^{2\ell-1}\unorm^2.\notag
\end{align}
\par
 Finally we have, for any $D\in\cT_h$ and $F\in\FD$,
\begin{equation}\label{eq:3DBdryEst3}
  \|u-\ID u\|_{L_\infty(F)}^2=\|u-\IF u\|_{L_\infty(F)}^2\lesssim
  \hF^{2\ell-1}|u|_{H^{\ell+\frac12}(F)}^2
   \lesssim \hD^{2\ell-1}|u|_{H^{\ell+1}(D)}^2
\end{equation}
 by \eqref{eq:HalfTrace}, \eqref{eq:1DInterpolationErrorHalf} and
 \eqref{eq:ConsistentID}.
\par
  The estimate \eqref{eq:InftyBdryEst} then follows from
  \eqref{eq:3DBdryEst1}--\eqref{eq:3DBdryEst3} and the triangle inequality.
\end{proof}
\par
 We also have estimates for
 the computable approximate solutions $\PO u_h$ and
 $\PZ u_h$.
\begin{theorem}\label{thm:3DMaxNormBdryEst}
     Assuming $\cT_h$ is quasi-uniform and the solution $u$ of \eqref{eq:Poisson} belongs to
 $H^{\ell+1}(\O)$ for $\ell$ between $1$ and $k$, there exists a positive constant
 $C$, depending only on $\rho$, $N$, $\gamma$ and $k$, such that
\begin{equation}\label{eq:3DLInftyError}
   \|u-\PO u_h\|_{L_\infty(\O)}+\|u-\PZ u_h\|_{L_\infty(\O)}\leq C
   \beta_h h^{\ell-(1/2)}\unorm,
\end{equation}
 where $\beta_h$ is defined in \eqref{eq:lambdah}.
\end{theorem}
\begin{proof}
   For any $D\in\cT_h$, we have,  by \eqref{eq:Sobolev},
 \eqref{eq:DiscreteEstimate1} and \eqref{eq:3DPODStability1},
\begin{align*}
 &\|u-\POD u_h\|_{\LinfD}\\
 &\hspace{40pt}\lesssim
 \hD^{-\frac32}\|u-\POD u_h\|_\LTD+\hD^{-\frac12}|u-\POD u_h|_\HOD
 +\hD^\frac12|u-\POD u_h|_\HTD\\
 &\hspace{40pt}\lesssim
 \hD^{-\frac32}\|u-\POD u_h\|_\LTD+\hD^{-\frac12}|u-\POD u_h|_\HOD
 +\hD^\frac12|u-\POD u|_\HTD\\
    &\hspace{70pt}+\hD^\frac12|\POD(u-u_h)|_\HTD\\
     &\hspace{40pt}\lesssim
 \hD^{-\frac32}\|u-\POD u_h\|_\LTD+\hD^{-\frac12}|u-\POD u_h|_\HOD
 +\hD^\frac12|u-\POD u|_\HTD\\
 &\hspace{60pt}+\hD^{-\frac12}|u-u_h|_\HOD,
\end{align*}
 which together with \eqref{eq:3DHTwoPODPDZ},
 Theorem~\ref{thm:3DComputableEnergyError}, Theorem~\ref{thm:3DLTwoErrors}
 and the quasi-uniformity of $\cT_h$ implies the estimate for $u-\PO u_h$.
\par
 Similarly we have, by \eqref{eq:Sobolev},
 \eqref{eq:DiscreteEstimate1} and
 \eqref{eq:3DPDZHOne},
\begin{align*}
 &\|u-\PDZ u_h\|_{\LinfD}\\
  &\hspace{40pt}\lesssim
 \hD^{-\frac32}\|u-\PDZ u_h\|_\LTD+\hD^{-\frac12}|u-\PDZ u_h|_\HOD
 +\hD^\frac12|u-\PDZ u|_\HTD\\
 &\hspace{70pt}+\hD^{-\frac12}|u-u_h|_\HOD,
\end{align*}
 which together with \eqref{eq:3DHTwoPODPDZ},
 Theorem~\ref{thm:3DComputableEnergyError},
 Theorem~\ref{thm:3DLTwoErrors}
 and the quasi-uniformity of $\cT_h$ implies the estimate for $u-\PZ u_h$.
\end{proof}
%%%%%%%%%%%%%%%%%%%%%%%%%%%%
\section{Concluding Remarks}\label{sec:Conclusions}
 We have developed error estimates for virtual element methods for the model
 Poisson problem in two and three dimensions that provide justifications for
  existing numerical results
 for polygonal (or polyhedral) meshes with small edges (or faces).
\par
 For the two dimensional problem, the convergence of the virtual element
 method based on the stabilizing bilinear form $S^D_2(\cdot,\cdot)$ is optimal under the
 shape regularity assumptions in Section~\ref{subsec:GlobalShapeRegularity}.
 Under the additional assumption that the edges of any subdomain
 in a polygonal mesh are comparable to one another,
 convergence of the virtual element method based on
 the stabilizing bilinear form $S^D_1(\cdot,\cdot)$ is also optimal.
\par
 For the three dimensional problem, the convergence of the virtual element method
  is optimal if, in addition to the assumptions in Section~\ref{subsec:Shape3D}, we
  also assume that the edges of any face in the polyhedral mesh are comparable to one another.
\par
  The results in this paper can be extended to virtual element methods with the
  stabilizing bilinear form
\begin{equation*}
  S^D(w,v)=(\PDZmt w,\PDZmt v)+\sum_{p\in\NPD}w(p)v(p)
\end{equation*}
 in two dimensions, and  the  stabilizing bilinear form
\begin{align*}
    S^D(w,v)&=(\PDZmt w,\PDZmt v)+\hD\SumF\Big(\hF^{-2}(\PFZmt w,\PFZmt v)_\LTF+
   \sum_{p\in\NPF}w(p)v(p)\Big).\label{eq:3DSD}
\end{align*}
 in three dimensions.  The stability for these virtual element methods is automatic and the
 error analysis also does not pose any new difficulties.
\par
 The results in this paper can also be extended to virtual element methods ($k\geq2$)
 where the inner product \eqref{eq:InnerProduct} is replaced by the inner product
\begin{equation*}
 (\!(\zeta,\eta)\!)=(\nabla \zeta,\nabla \eta)
     +\Big(\int_{D}\zeta\,dx\Big)\Big(\int_{D} \eta\,dx\Big).
\end{equation*}
\par
 We note that error estimates for the Poisson problem on
 general polygonal or polyhedral domains can also
 be obtained by the techniques developed in this paper.
\par
 Finally it would be interesting to construct a three dimensional
 analog of the stabilizing bilinear form
 $S^D_2(\cdot,\cdot)$ defined in Section~\ref{subsec:DiscreteProblem}
  so that the convergence of the virtual element methods is optimal
 for polyhedral meshes with arbitrarily small faces and edges, and
 $L_\infty$ error estimates can be established without assuming the meshes are quasi-uniform.
  We conjecture that such a bilinear form can be defined by
\begin{align*}
  S^D(v,w)&=\SumF\hF(\GF\POF v,\GF\POF w)_\LTF\\
   &\hspace{40pt}+\SumF\hF\sum_{e\in\EF}
   \hE\big(\p (v-\POF v)/\p s,\p (w-\POF w)/\p s\big)_{L_2(e)}.
\end{align*}

\end{document}